\documentclass{article}

\usepackage{geometry}
\usepackage{graphicx} 
\usepackage{epstopdf} 

\usepackage[colorlinks, hypertexnames=false]{hyperref}
\hypersetup{linkcolor=blue, citecolor=red} 

\usepackage{bm}
\usepackage{amsmath} 
\usepackage{amssymb} 
\usepackage{stmaryrd} 
\usepackage{multirow}

\usepackage{amsthm} 
\usepackage{mathtools} 

\usepackage[dvipsnames]{xcolor}
\definecolor{red}{rgb}{1.00,0.00,0.00}
{\numberwithin{equation}{section}
\setlength{\parindent}{1em}

\newtheorem{theorem}{Theorem}[section]
\newtheorem{lemma}{Lemma}[section]
\newtheorem{remark}{Remark}[section]

\newcommand{\normmm}[1]{{\left\vert\kern-0.25ex\left\vert
\kern-0.25ex\left\vert #1
    \right\vert\kern-0.25ex\right\vert\kern-0.25ex\right\vert}}

\geometry{left=3cm,right=3cm,top=3cm,bottom=2cm}

\begin{document}           

\title{Constraint energy minimizing generalized multiscale finite element method for convection diffusion equation}

\author{Lina Zhao\footnotemark[1]\qquad
    \;Eric Chung\footnotemark[2]}
\renewcommand{\thefootnote}{\fnsymbol{footnote}}
\footnotetext[1]{Department of Mathematics, City University of Hong Kong, Hong Kong SAR, China. ({linazha@cityu.edu.hk}).}
\footnotetext[2]{Department of Mathematics, The Chinese University of Hong Kong, Hong Kong SAR, China. ({tschung@math.cuhk.edu.hk}).}

\maketitle

\begin{abstract}
In this paper we present and analyze a constraint energy minimizing generalized multiscale finite element method for convection diffusion equation. To define the multiscale basis functions, we first build an auxiliary multiscale space by solving local spectral problems motivated by analysis. Then constraint energy minimization performed in oversampling domains is exploited to construct the multiscale space. The resulting multiscale basis functions have a good decay property even for high contrast diffusion and convection coefficients. Furthermore, if the number of oversampling layer is chosen properly, we can prove that the convergence rate is proportional to the coarse mesh size. Our analysis also indicates that the size of the oversampling domain weakly depends on the contrast of the heterogeneous coefficients. Several numerical experiments are presented illustrating the performances of our method.

\end{abstract}

\textbf{Keywords:}
Multiscale method, convection diffusion equation, local multiscale basis function, local spectral problem

\pagestyle{myheadings} \thispagestyle{plain}
\markboth{ZhaoChung} {CEM-GMsFEM for convection diffusion}

\section{Introduction}

In this paper we consider the following convection diffusion equation: Find $u\in H^1_0(\Omega)$ such that
\begin{equation}
\begin{split}
-\nabla \cdot (\kappa \nabla u) + \bm{\beta}\cdot \nabla u& = f\quad \mbox{in}\;\Omega,\\
u&=0 \quad \mbox{on}\;\partial \Omega,
\end{split}
\label{eq:model}
\end{equation}
where $\Omega\subset \mathbb{R}^2$ is the computational domain and $\bm{\beta}\in L^{\infty}(\Omega)^2$. We assume that both $\kappa$ and $\bm{\beta}$ are heterogeneous coefficients with multiple scales and very high contrast, in addition, the velocity field $\bm{\beta}$ is incompressible, i.e., $\nabla\cdot\bm{\beta}=0$. Further, we assume that there exist $\kappa_0,\kappa_1$ such that $\kappa_0\leq \kappa\leq \kappa_1$, where $\kappa_1/\kappa_0$ could be large. Moreover, we let $\beta_1$ and $\beta_0$ represent the supremum and minimum of $|\bm{\beta}|$ over $\Omega$, respectively. For simplicity, we assume that $\beta_0\geq 1$ and $\kappa_0 \geq 1$.

There are a large number of works devoted to solving the convection diffusion equation \eqref{eq:model}. This problem becomes even harder to solve when the P\'{e}clet number is large. To overcome this issue, numerous multiscale methods have been developed such as variational multiscale method \cite{HughesSangalli07,john2006two,song2010variational,xie2021variational}, multiscale finite element method \cite{ParkHou04}, multiscale discontinuous Galerkin method \cite{KimWheeler2014,chung2013sub}, heterogeneous multiscale method \cite{henning2010heterogeneous}, 
variational multiscale stabilization \cite{Li17} and multiscale stabilization \cite{CaloChung16,ChungEfendievLeung20}. The aforementioned methods are based on special construction of basis functions, which typically resolve fine scale information on relatively coarse meshes.

In this paper, our purpose is to study constraint energy minimizing generalized multiscale finite element method (CEM-GMsFEM) for convection diffusion equation. CEM-GMsFEM is based on GMsFEM \cite{EfendievGalvis13,ChungEfendiev14,ChungLi14,EfendievGalvisLi14,ChungEfendievHou16,ChungLee19,Chen20} and have been successfully applied to a wide range of partial differential equations \cite{ChungEfendievmixed18,LiChung19,Vasilyeva19,Cheung20,Cheungdual20,Fu20,FuChung20}.
The key steps of CEM-GMsFEM used in this paper can be summarized as follows. First, we need to build an auxiliary space. Specially, we define a suitable spectral problem over each coarse cell, and the first few eigenfunctions corresponding to small eigenvalues that contain important features about the multiscale coefficients $\kappa$ and $\bm{\beta}$ are used in the definition of the local auxiliary multiscale space. Second, we solve an appropriate energy subject to some constraints over the oversampling domain by using the local auxiliary multiscale space. We emphasize that the choice of spectral problem is very important and can ensure the good performances of the method. We prove that the multiscale basis functions are localizable. In addition, we prove the convergence rate $H/\Lambda$ if the number of oversampling layer is chosen suitably, where $\Lambda$ is the minimal eigenvalue that the corresponding eigenvector is not included in the space. Our analysis also shows that the size of the oversampling domain depends on the contrast of the heterogeneous coefficients weakly (logarithmically). We present several numerical experiments to verify the performances of CEM-GMsFEM. In particular, we exploit one example where velocity is obtained by solving Darcy law with SPE benchmark heterogeneous field.

The rest of the paper is organized as follows. In the next section, we provide some preliminaries. Then in section~\ref{sec:basis}, we present in detail the construction of our multiscale basis functions. Specially, we introduce the spectral problem that is used to define the auxiliary multiscale basis and the energy minimization that will be used to construct the multiscale space. The decay property of the multiscale basis function and the error estimates are presented in section~\ref{sec:analysis}. Several numerical experiments are carried out in section~\ref{sec:numerical} to test the performances of our method. Finally, a conclusion is given.

\section{Preliminaries}

In this section we introduce some notations that will be used throughout the paper. Let $\mathcal{T}_H$ be a conforming partition of $\Omega$ into rectangular elements. Here $H$ is the coarse meshsize and this partition is called coarse grid. We let $N_c$ be the number of vertices and $N$ be the number of coarse grids. We assume that each coarse element is partitioned into a connected union of fine-grid blocks and this partition is denoted as $\mathcal{T}_h$. Note that $\mathcal{T}_h$ is a refinement of the coarse grid $\mathcal{T}_H$ with the meshsize $h$. It is assumed that the fine grid is fine enough to resolve the solution. Here we use triangular grid as the fine grid and extension to other shapes of grids such as quadrilateral is straightforward. For each coarse element, we define an oversampling domain $K_{i,m}\subset \Omega$ by enlarging $K_i$ by $m$ coarse grid layers, where $m\geq 1$ is an integer, see Figure~\ref{fig:grid} for an illustration of the coarse grid, fine grid and oversampling domain.

\begin{figure}[t]
\centering
\includegraphics[width=0.35\textwidth]{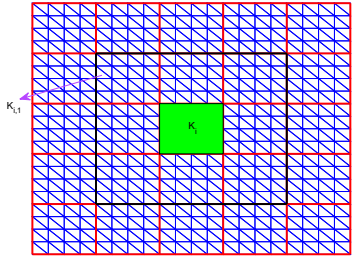}
\caption{An illustration of the fine grid, coarse grid and oversampling domain.}
\label{fig:grid}
\end{figure}

The weak solution of \eqref{eq:model} reads as follows: Find $u\in H^1_0(\Omega)$ such that
\begin{align}
A(u,v)=(f,v)\quad \forall v\in H^1_0(\Omega),\label{eq:weak}
\end{align}
where $A(u,v)=(\kappa \nabla u, \nabla v)+(\bm{\beta}\cdot\nabla u,v)$.

\section{Construction of multiscale basis function}\label{sec:basis}

%

In this section we present the construction of the multiscale basis functions. To this end, we first define the auxiliary space which is obtained by solving local spectral problems. Then we solve an appropriate energy subject to some constraints over the oversampling domain, which delivers the multiscale space.

Let $K_i$ be the $i$-th coarse block and let $V(K_i)$ be the restriction of $H^1_0(\Omega)$ on $K_i$. We define a local spectral problem which is defined as follows: Find $\lambda_j^{(i)}\in \mathbb{R}$ and $\phi_j^{(i)}\in V(K_i)$ such that
\begin{equation}
a_i(\phi_j^{(i)},v) = \lambda_j^{(i)} s_i(\phi_j^{(i)},v)\quad \forall v\in V(K_i).\label{eq:spectral}
\end{equation}
We remark that the above problem is solved on the fine mesh in the actual computations. According to our analysis, we can choose
\begin{align}
a_i(u,v) = \int_{K_i} \kappa\nabla u\cdot\nabla v\;dx,\quad s_i(u,v)=\int_{K_i} H^{-2}\kappa|\bm{\beta}|^2 uv\;dx.\label{eq:local-form}
\end{align}
We let $\lambda_j^{(i)}$ be the eigenvalues of \eqref{eq:spectral} arranged in ascending order. We will use the first $l_i$ eigenfunctions to construct our local auxiliary multiscale space $V_{aux}^{(i)}$, where $V_{aux}^{(i)}=\text{span}\{\phi_j^{(i)}| j \leq l_i\}$. The global auxiliary multiscale space $V_{aux}$ is the sum of these local auxiliary multiscale space, namely $V_{aux}=\oplus_{i=1}^N V_{aux}^{(i)}$. We will use this space to construct the multiscale basis functions which are $\phi$-orthogonal to the auxiliary space defined above.


For the local auxiliary space $V_{aux}^{(i)}$, the bilinear form $s_i$ in \eqref{eq:local-form} defines an inner product with norm $\|v\|_{s(K_i)}=s_i(v,v)^{1/2}$. These local inner products and norms provide natural definitions of inner product and norm for the global auxiliary multiscale space $V_{aux}$, which are defined by
\begin{align*}
s(v,w)=\sum_{i=1}^Ns_i(v,w),\quad \|v\|_s=s(v,v)^{1/2}\quad \forall v,w\in V_{aux}.
\end{align*}
We note that $s(v,w)$ and $\|v\|_s$ are also an inner product and norm for the space $H^1_0(\Omega)$. Using the above inner product, we can define the notion of $\phi$-orthogonality in the space $H^1_0(\Omega)$ (cf. \cite{ChungEfendievleung18}). Given a function $\phi_j^{(i)}\in V_{aux}$, we say that a function $\psi\in H^1_0(\Omega)$ is $\phi_j^{(i)}$-orthogonal if
\begin{align*}
s(\psi,\phi_j^{(i)})=1,\quad s(\psi,\phi_{j'}^{(i')})=0\quad \mbox{if} \;j'\neq j,i'\neq i.
\end{align*}
Now, we define $\pi_i:L^2(K_i)\rightarrow V_{aux}^{(i)}$ to be the projection with respect to the inner product $s_i(v,w)$. More precisely, it is defined by
\begin{align*}
\pi_i(u)=\sum_{j=1}^{l_i}\frac{s_i(u,\phi_j^{(i)})}{s_i(\phi_j^{(i)},\phi_j^{(i)})}\phi_j^{(i)}\quad \forall u\in H^1_0(\Omega).
\end{align*}
In addition, we let $\pi$ be the projection with respect to the inner product $s(v,w)$. Hence, $\pi$ is defined by
\begin{align*}
\pi(u)=\sum_{i=1}^N\sum_{j=1}^{l_i}\frac{s_i(u,\phi_j^{(i)})}{s_i(\phi_j^{(i)},\phi_j^{(i)})}\phi_j^{(i)}\quad \forall u\in H^1_0(\Omega),
\end{align*}
which satisfies $\pi(u)=\sum_{i=1}^N\pi_i(u)$.

We now present the construction of our multiscale basis functions. For each coarse element $K_i$ and an oversampling domain $K_{i,m}\subset \Omega$, we define the multiscale basis function $\psi_{j,ms}^{(i)}\in V_0(K_{i,m})$ by
\begin{align}
\psi_{j,ms}^{(i)}=\mbox{argmin}\{a(\psi,\psi)\;|\;\psi\in V_0(K_{i,m}),\quad \psi\; \mbox{is}\; \phi_j^{(i)}-\mbox{orthogonal}\},\label{eq:local-minimization}
\end{align}
where $V(K_{i,m})$ is the restriction of $H^1_0(\Omega)$ in $K_{i,m}$ and $V_0(K_{i,m})$ is the subspace of $V(K_{i,m})$ with zero trace on $\partial K_{i,m}$. Our multiscale finite element space $V_{ms}$ is defined by
\begin{align*}
V_{ms}=\mbox{span}\{\psi_{j,ms}^{(i)}\;|\; 1\leq j\leq l_i,1\leq i\leq N\}.
\end{align*}
The existence of the solution of the minimization problem \eqref{eq:local-minimization} will be proved in Lemma~\ref{lemma:vaux}. Moreover, we illustrate the importance of the auxiliary space on the decay of the multiscale basis functions. Here we take $H=1/10$, $\kappa=1$ and $\bm{\beta}=(\cos(18\pi y)\sin(18\pi x), -\cos(18\pi x)\sin(18\pi y))^T$, where $\bm{\beta}$ is a highly oscillatory vectorized function and its profile is shown in Figure~\ref{plot:beta}. In Figure~\ref{plot:multiscale-basis}, we display the first four nonzero eigenvalues obtained from solving the local spectral problem \eqref{eq:spectral}, and a multiscale basis function with one eigenfunction and four eigenfunctions in the local auxiliary space. We can observe that if enough number of eigenfunctions are exploited in solving the energy minimization problem, then the multiscale basis functions have a fast decay outside of the coarse block.

With the above preparations, the multiscale solution $u_{ms}$ is defined as the solution of the following problem: Find $u_{ms}\in V_{ms}$ such that
\begin{align}
A(u_{ms},v)=(f,v)\quad \forall v\in V_{ms}.\label{eq:multiscale-solution}
\end{align}

\begin{remark}

The minimization problem \eqref{eq:local-minimization} is implicit, we can recast it into the explicit form by introducing the lagrange multiplier. The equivalent explicit form reads as follows: Find $\psi_{j,ms}^{(i)}\in V_0(K_{i,m}),\lambda\in V_{aux}^{(i)}(K_{i,m})$ such that
\begin{align*}
a(\psi_{j,ms}^{(i)}, p)+s(p,\lambda)&=0\quad \forall p\in V_0(K_{i,m}),\\
s(\psi_{j,ms}^{(i)}-\phi_j^{(i)},q)&=0\quad \forall q\in V_{aux}^{(i)}(K_{i,m}),
\end{align*}
where $V_{aux}^{(i)}(K_{i,m})$ is the union of all local auxiliary spaces for $K_j\subset K_{i,m}$. One can numerically solve the above continuous problem on fine scale mesh.

\begin{figure}[t]
\centering
\includegraphics[width=0.35\textwidth]{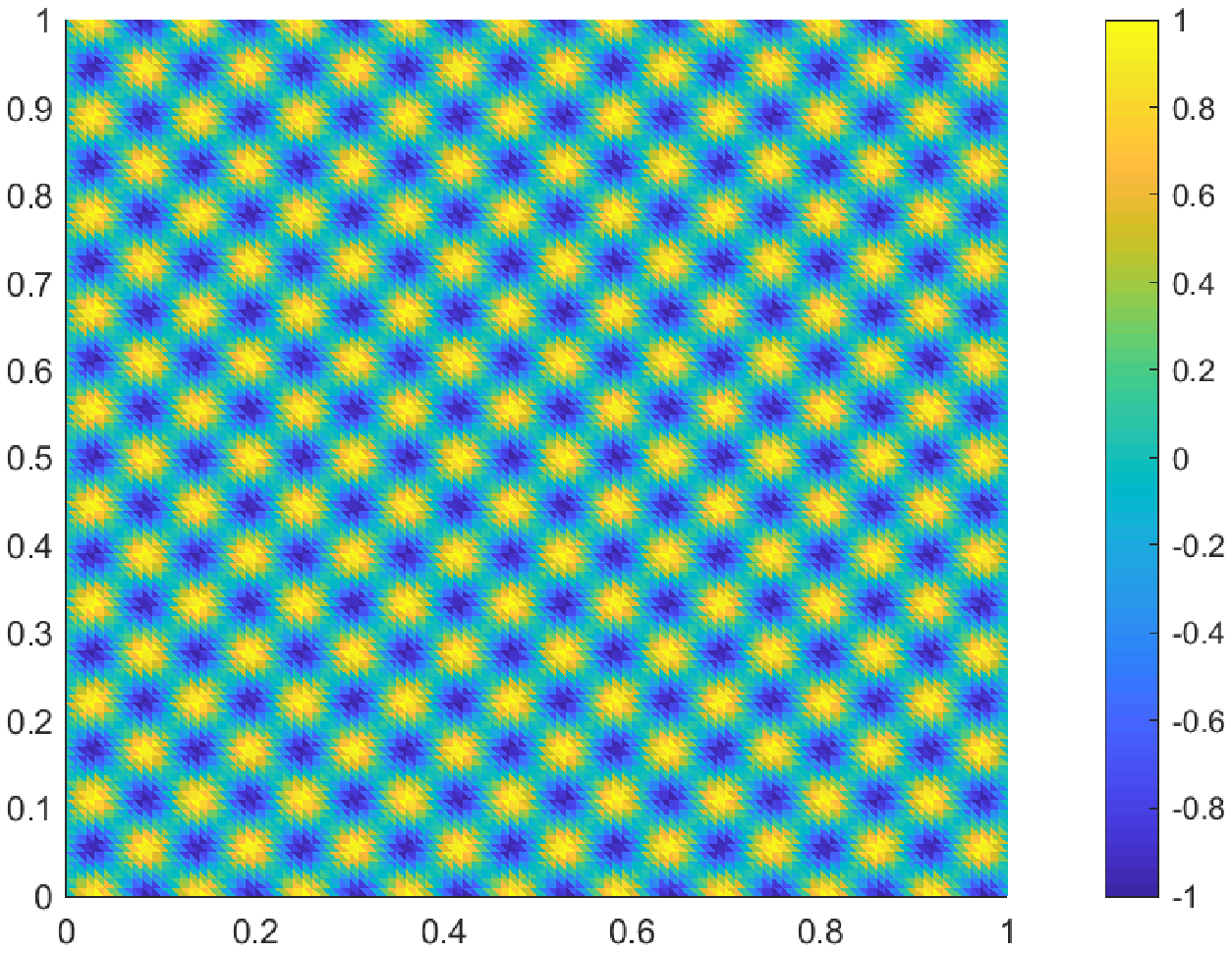}
\includegraphics[width=0.35\textwidth]{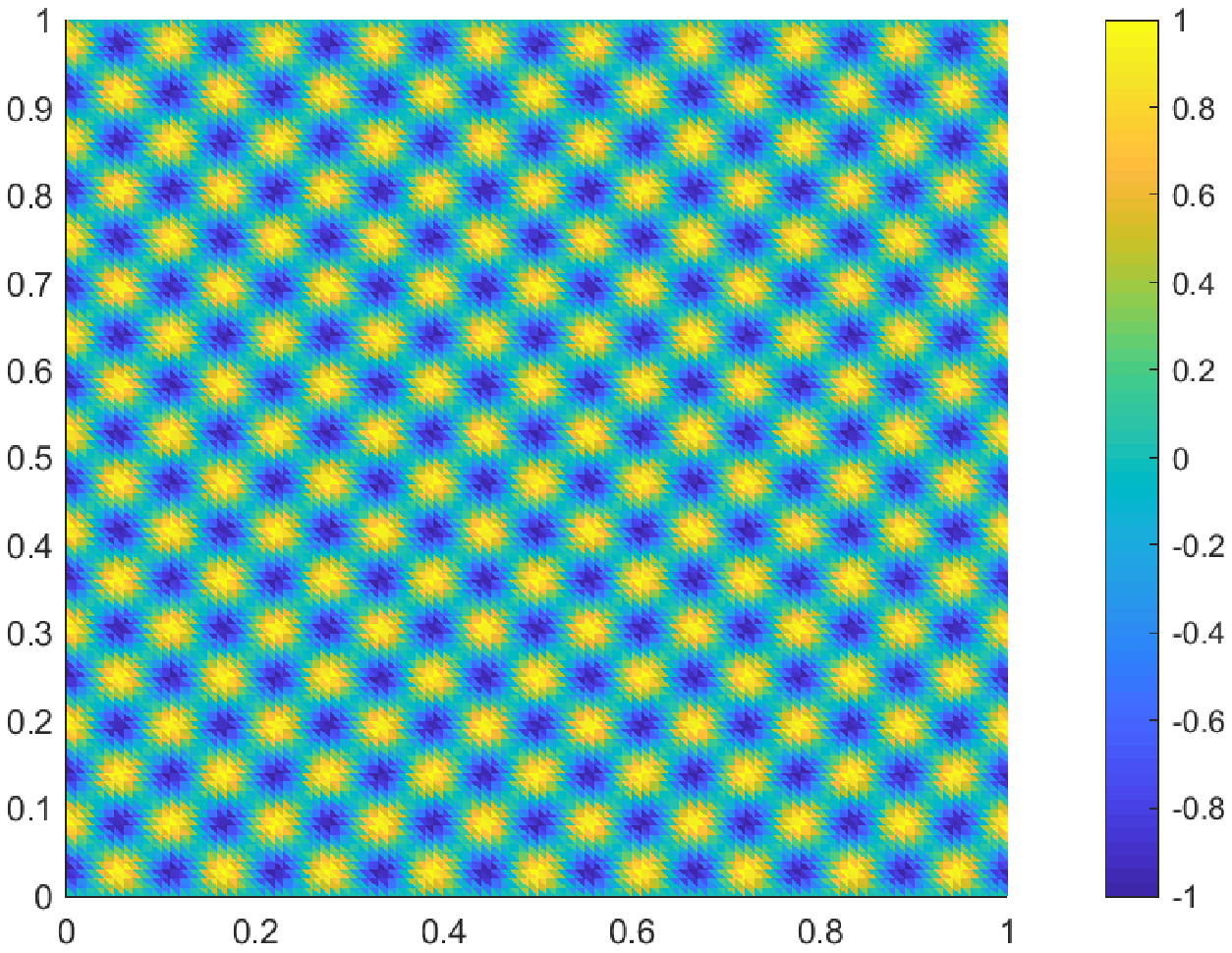}
\caption{Profile of the first component of $\bm{\beta}$ (left) and the second component of $\bm{\beta}$ (right).}
\label{plot:beta}
\end{figure}

\begin{figure}[t]
\centering
\includegraphics[width=0.32\textwidth]{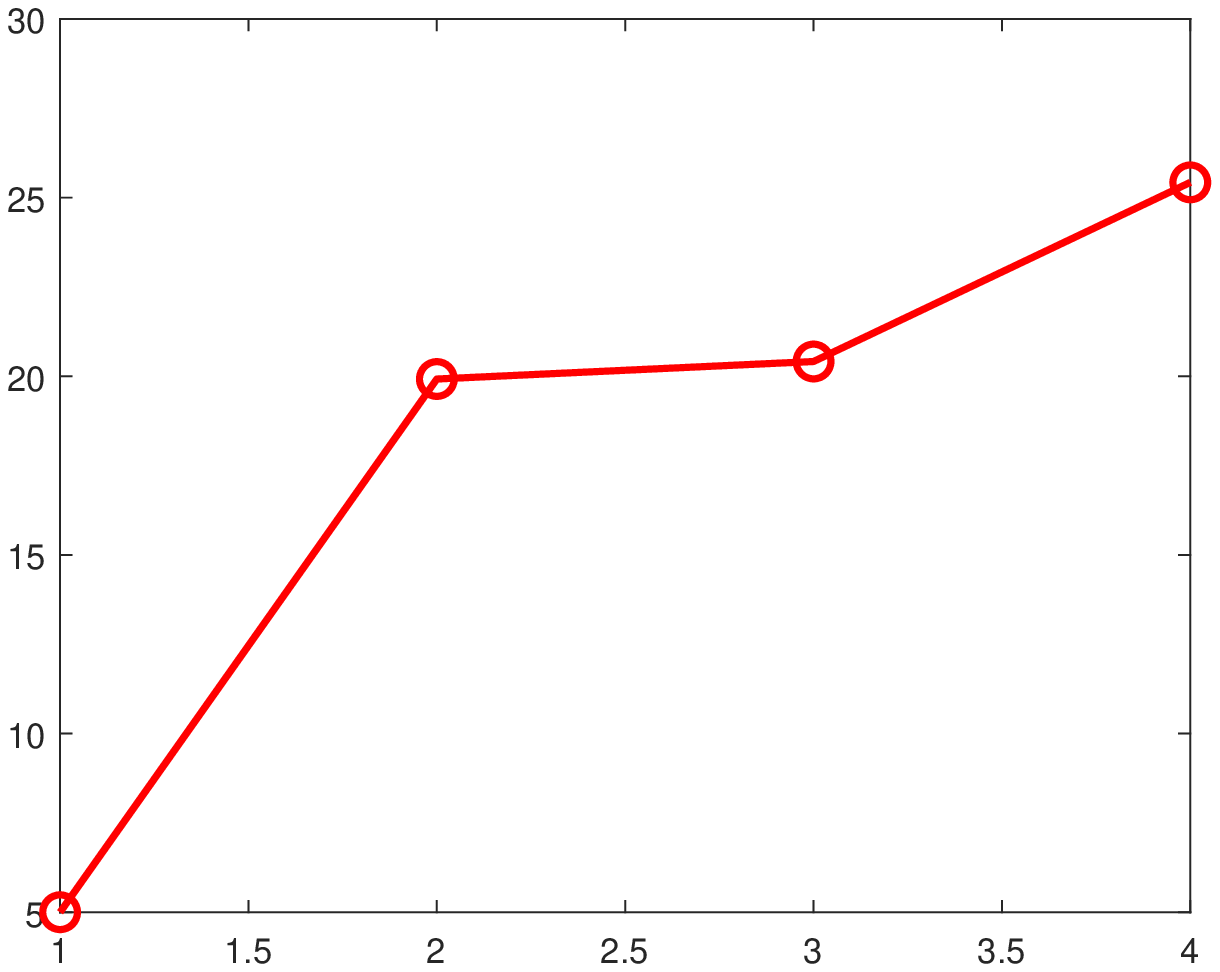}
\includegraphics[width=0.32\textwidth]{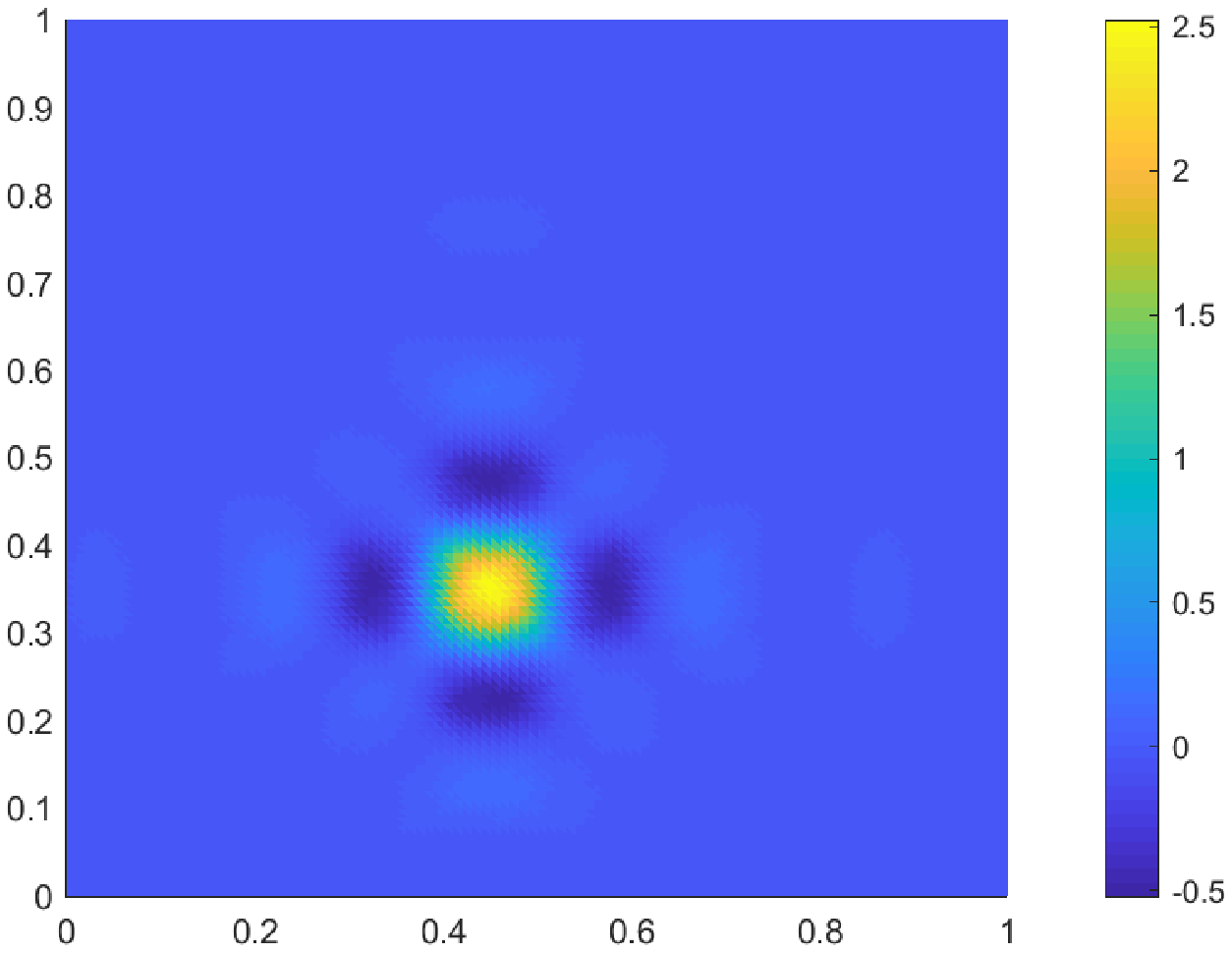}
\includegraphics[width=0.32\textwidth]{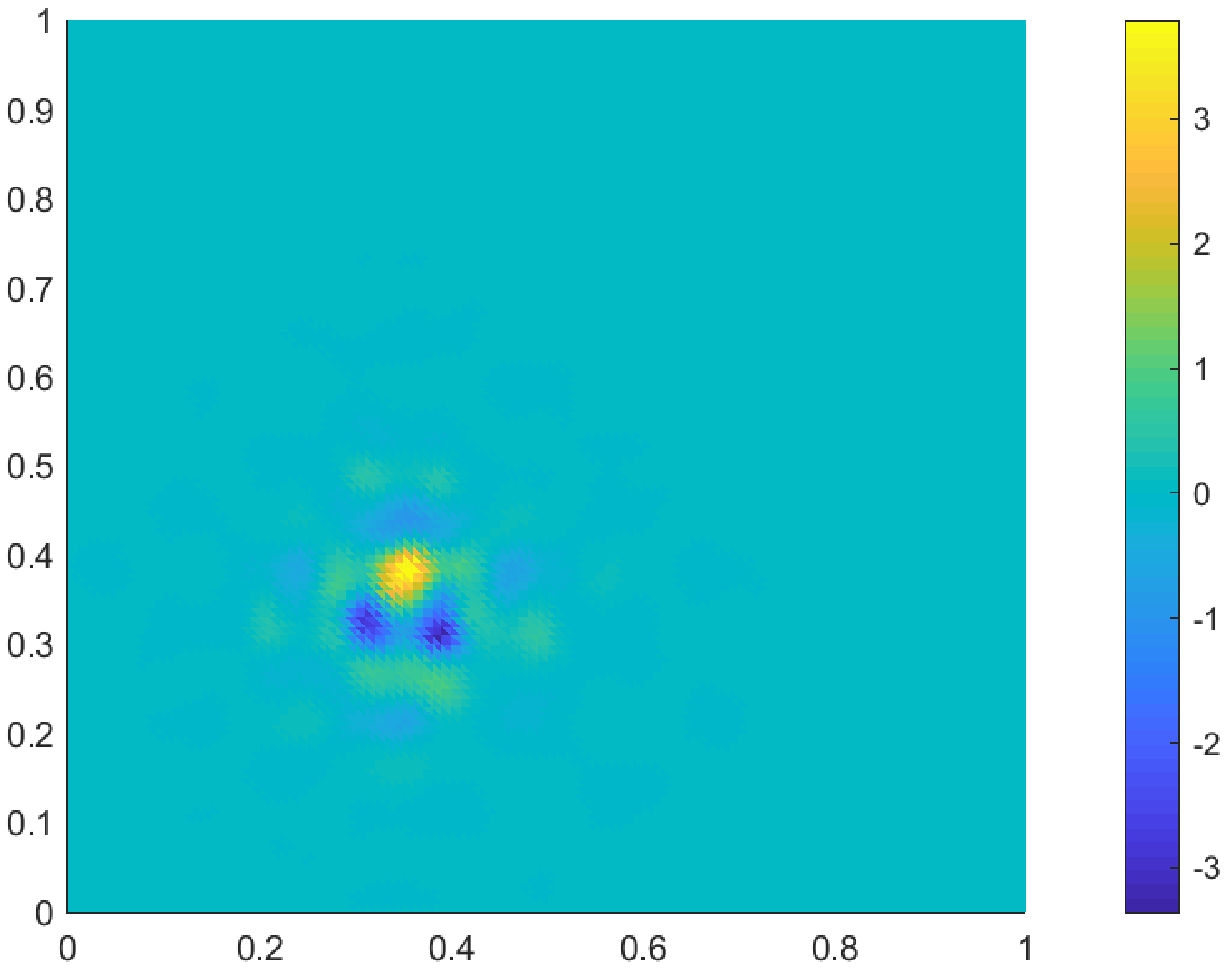}
\caption{A plot of the first four nonzero eigenvalues (left), a multiscale basis function using one eigenfunction in each local auxiliary space (middle) and a multiscale basis function using 4 eigenfunctions in each local auxiliary space (right).}
\label{plot:multiscale-basis}
\end{figure}

The local multiscale basis construction is motivated by the global basis construction defined below. Since the global multiscale basis functions will be exploited for later analysis, we present its construction here. The global multiscale basis function $\psi_j^{(i)}\in H^1_0(\Omega)$ is defined by
\begin{align}
\psi_{j}^{(i)}=\textnormal{argmin}\{a(\psi,\psi)\;|\;\psi\in H^1_0(\Omega),\quad \psi\; \mbox{is}\; \phi_j^{(i)}-\textnormal{orthogonal}\},\label{eq:global-minimization}
\end{align}
Thereby, the global multiscale finite element space $V_{glo}$ is defined by
\begin{align*}
V_{glo}=\textnormal{span}\{\psi_j^{(i)}\;|\; 1\leq j\leq l_i,1\leq i\leq N\}.
\end{align*}

\end{remark}

To facilitate later analysis, we define the projection $\tilde{u}\in V_{glo}$ of $u$ by
\begin{equation}
a(\tilde{u},v) = a(u,v) \quad\forall v\in V_{glo},\label{eq:projection}
\end{equation}
where $a(u,v)=\int_{\Omega}\kappa \nabla u \cdot \nabla v\;dx$. The above problem is well posed thanks to the Lax-Milgram lemma.
Notice that we have $V = V_{glo} \oplus \tilde{V}$, where $\tilde{V}$ is the kernel of the projection $\pi: L^2 \rightarrow V_{aux}$ with respect to the inner product $s(\cdot, \cdot)$.
Interested readers can refer to \cite{ChungEfendievleung18} for more discussions regarding this.
We can infer from \eqref{eq:projection} that $u-u_{glo}\in V_{glo}^\perp=\tilde{V}$, thereby $\pi(u-\tilde{u})=0$.
%
%
%
%
%

\section{Analysis}\label{sec:analysis}

In this section we first prove the convergence for the projection defined in \eqref{eq:projection}, which will be served as an intermediate tool for the convergence analysis of the multiscale solution. Then we show the decay property of the multiscale basis functions. Finally, we prove the convergence of the multiscale solution. Before proving the convergence of the proposed method, we introduce some notations that will be used later. We define $a$-norm by $\|u\|_a^2=\int_{\Omega}\kappa |\nabla u|^2\;dx$. For a given subdomain $\Omega_i\subset \Omega$, we define the local $a$-norm and $s$-norm by $\|u\|_{a(\Omega_i)}^2=\int_{\Omega_i}\kappa |\nabla u|^2\;dx$ and $\|u\|_{s(\Omega_i)}^2=H^{-2}\int_{\Omega_i}\kappa |\bm{\beta}|^2 u^2\;dx$.

\begin{lemma}\label{lemma:projection}
Let $u$ be the solution of \eqref{eq:weak} and $\tilde{u}$ satisfy \eqref{eq:projection}. We have $u-\tilde{u}\in \tilde{V}$ and
\begin{equation*}
\| u -\tilde{u}\|_a \leq H\Lambda^{-\frac{1}{2}}\Big(\|\kappa^{-\frac{1}{2}}  |\bm{\beta}|^{-1} f\|_{L^2(\Omega)}
+\|\kappa^{-\frac{1}{2}}\nabla u\|_{L^2(\Omega)}\Big),
\end{equation*}
where $\Lambda=\min_{1\leq i\leq N}\lambda_{l_i+1}^{(i)}$.

\end{lemma}

\begin{proof}

We have from \eqref{eq:projection} and the Cauchy-Schwarz inequality that
\begin{equation*}
\begin{split}
\|u-\tilde{u}\|_a^2 &= a(u-\tilde{u}, u-\tilde{u}) = a(u, u-\tilde{u}) = (f,u-\tilde{u}) - (\bm{\beta}\cdot \nabla u, u-\tilde{u}) \\
&\leq \|\kappa^{-\frac{1}{2}} |\bm{\beta}|^{-1} f\|_{L^2(\Omega)}\|\kappa^{\frac{1}{2}}  |\bm{\beta}| (u-\tilde{u})\|_{L^2(\Omega)}
+\|\kappa^{-\frac{1}{2}}\nabla u\|_{L^2(\Omega)}\|\kappa^{\frac{1}{2}}|\bm{\beta}|(u-\tilde{u})\|_{L^2(\Omega)}\\
& \leq (H\|\kappa^{-\frac{1}{2}}  |\bm{\beta}|^{-1} f\|_{L^2(\Omega)}+H\|\kappa^{-\frac{1}{2}}\nabla u\|_{L^2(\Omega)}) \| u-\tilde{u}\|_s.
\end{split}
\end{equation*}
Since $\pi(u-\tilde{u})=0$, we have from the spectral problem \eqref{eq:spectral} that
\begin{align*}
\|u-\tilde{u}\|_s^2=\sum_{i=1}^N \|u-\tilde{u}\|_{s(K_i)}^2=\sum_{i=1}^N \|(I-\pi_i)(u-\tilde{u})\|_{s(K_i)}^2\leq \frac{1}{\Lambda}\|u-\tilde{u}\|_a^2.
\end{align*}
Therefore, the following estimate holds
\begin{equation*}
\| u -\tilde{u}\|_a \leq H\Lambda^{-\frac{1}{2}}\Big(\|\kappa^{-\frac{1}{2}} |\bm{\beta}|^{-1} f\|_{L^2(\Omega)}
+\|\kappa^{-\frac{1}{2}}\nabla u\|_{L^2(\Omega)}\Big).
\end{equation*}

\end{proof}


The above lemma shows the convergence of the projection (cf. \eqref{eq:projection}) defined by using global multiscale basis functions. Next, we are going to show that the global multiscale basis functions are localizable. For this purpose, we introduce some concepts that will be used later. For each coarse block $K_i$, we define $B$ to be a bubble function and $B\mid_{\tau}=\frac{\varphi_1\varphi_2\varphi_3}{27},\forall \tau\in \mathcal{T}_h(K_i)$, where $\varphi_i$ is barycentric coordinate and $\mathcal{T}_h(K_i)$ denotes the collection of fine grids restricted to $K_i$, and more information regarding the bubble function $B$ can be found in \cite{Verfurth96}.
We define the constant by
\begin{align*}
C_{\pi} = \sup_{K_i\in \mathcal{T}_H,\mu \in V_{aux}}\frac{\int_{K_i} |\bm{\beta}|^2\mu^2\;dx}{\int_{K_i}|\bm{\beta}|^2 B \mu^2\;dx}.
\end{align*}
The following lemma considers the following minimization problem defined on a coarse block $K_i$
\begin{align}
v=\text{argmin}\{a(\psi,\psi)\;|\; \psi\in V_0(K_i),\quad s_i(\psi, v_{aux})=1,\quad s_i(\psi, w)=0 \quad \forall w\in V_{aux}^{\perp}\}\label{eq:minimization}.
\end{align}
for a given $v_{aux}\in V_{aux}^{(i)}$ with $\|v_{aux}\|_{s(K_i)}=1$, where $v_{aux}^\perp\subset V_{aux}^{(i)}$ is the orthogonal complement of $\text{span}\{v_{aux}\}$ with respect to the inner product $s_i$.

The next lemma shows the existence of the solution to the minimization problem \eqref{eq:minimization}, which follows similar line to that of \cite{ChungEfendievleung18}. We also provide the proof here for the readers' convenience.
\begin{lemma}\label{lemma:vaux}
For all $v_{aux}\in V_{aux}$ there exists a function $v\in H^1_0(\Omega)$ such that
\begin{align*}
\pi(v) = v_{aux},\quad  \|v\|_a^2\leq C \|v_{aux}\|_s^2,\quad supp(v)\subset supp(v_{aux}).
\end{align*}

\end{lemma}

\begin{proof}
Let $v_{aux}\in V_{aux}^{(i)}$. The minimization problem \eqref{eq:minimization} is equivalent to the following variational problem: Find $v\in V_0(K_i)$ and $\mu \in V_{aux}^{(i)}$ such that
\begin{align}
a_i(v,w)+s_i(w,\mu) &=0 \quad \forall w\in V_0(K_i),\label{eq:saddle1}\\
s_i(v, \phi) & = s_i(v_{aux}, \phi)\quad \forall \phi\in V_{aux}^{(i)}.\label{eq:saddle2}
\end{align}
Note that, the well-posedness of the minimization problem \eqref{eq:minimization} is equivalent to the existence of a function $v\in V_0(K_i)$ such that
\begin{align*}
s_i(v,v_{aux})\geq C\|v_{aux}\|_s^2,\quad \|v\|_{a(K_i)}\leq C \|v_{aux}\|_{s(K_i)},
\end{align*}
where $C$ is independent of the meshsize but possibly depends on the problem parameters.

Note that $v_{aux}$ is supported in $K_i$. We let $v=Bv_{aux}$, it then follows from the definition of $s_i$ that
\begin{align*}
s_i(v,v_{aux})=H^{-2}\int_{K_i}\kappa |\bm{\beta}|^2 B v_{aux}^2\geq C_{\pi}^{-1} \|v_{aux}\|_{s(K_i)}^2.
\end{align*}
Since $\nabla (Bv_{aux})=v_{aux}\nabla B+B \nabla v_{aux}$, $|B|\leq 1$ and $|\nabla B|^2\leq C H^{-2}$, we have
\begin{align*}
\|v\|_{a(K_i)}^2=\|Bv_{aux}\|_{a(K_i)}^2\leq C \|v\|_{a(K_i)}\Big(\|v_{aux}\|_{a(K_i)}+\|v_{aux}\|_{s(K_i)}\Big).
\end{align*}
Finally, using the spectral problem \eqref{eq:spectral}, we can obtain
\begin{align*}
\|v_{aux}\|_{a(K_i)}\leq (\max_{1\leq j\leq l_i }\lambda_j^{(i)})\|v_{aux}\|_{s(K_i)}.
\end{align*}
This proves the unique solvability of the minimization problem \eqref{eq:minimization}. So $v$ and $v_{aux}$ satisfy \eqref{eq:saddle1}-\eqref{eq:saddle2}. Then we can obtain $\pi_i(v)=v_{aux}$ from \eqref{eq:saddle2}. Therefore, the preceding arguments complete the proof.

\end{proof}

Next, we will show that the multiscale basis functions have a decay property. To this end,  we define the cutoff function with respect to the oversampling domains.
For each $K_i$, we recall that $K_{i,m}$ is the oversampling coarse region by
enlarging $K_i$ by $m$ coarse grid layers. For $M>m$, we
define $\chi_{i}^{M,m}\in \mbox{span}\{\chi_i^{ms}\}$ such that $0\leq \chi_i^{M,m}\leq 1$
and
\begin{align}
\chi_i^{M,m}&=1 \quad \mbox{in}\;K_{i,m},\label{cut1}\\
\chi_i^{M,m}&=0\quad \mbox{in}\;\Omega\backslash K_{i,M}\label{cut2}.
\end{align}
Note that we have $K_{i,m}\subset K_{i,M}$ and $\{\chi_{i}^{ms}\}_{i=1}^N$ are the standard  multiscale finite element (MsFEM) basis functions (cf. \cite{HouWu97}).

\begin{lemma}\label{lemma:decay}
We consider the oversampling domain $K_{i,k}$ with $k\geq 2$. That is, $K_{i,k}$ is an oversampling region by enlarging $K_i$ by $k$ coarse grid layers. Let $\phi_j^{(i)}$ be a given auxiliary multiscale basis function. Let $\psi_{j,ms}^{(i)}$ be the multiscale basis function achieved from \eqref{eq:local-minimization} and let $\psi_j^{(i)}$ be the global multiscale basis function obtained from \eqref{eq:global-minimization}. Then we have
\begin{align*}
\|\psi_j^{(i)}-\psi_{j,ms}^{(i)}\|_a^2\leq E \|\phi_j^{(i)}\|_{s(K_i)}^2,
\end{align*}
where $E=8D^2(1+\Lambda^{-1})(1+\frac{\Lambda^{1/2}}{2D^{1/2}})^{1-k}$.

\end{lemma}

\begin{proof}
For the given $\phi_j^{(i)}\in V_{aux}$, it follows from Lemma~\ref{lemma:vaux} that there exists a $\tilde{\phi}_j^{(i)}$ such that
\begin{align}
\pi(\tilde{\phi}_j^{(i)}) = \phi_j^{(i)},\quad \|\tilde{\phi}_j^{(i)}\|_a^2\leq C \|\phi_j^{(i)}\|_s^2\quad \mbox{and}\quad \text{supp}(\tilde{\phi}_j^{(i)})\subset K_i.\label{eq:phit}
\end{align}
We let $\eta = \psi_j^{(i)}-\tilde{\phi}_j^{(i)}$. Note that $\eta\in \tilde{V}$ since $\pi(\eta)=0$. By using the resulting variational forms of the minimization problems, we can obtain
\begin{align}
a(\psi_j^{(i)}, v)+s(v, \mu_j^{(i)})=0 \quad \forall v\in V\label{eq:eq1}
\end{align}
and
\begin{align}
a(\psi_{j,ms}^{(i)}, v)+s(v, \mu_{j,ms}^{(i)})=0 \quad \forall v\in V_0(K_{i,k})\label{eq:eq2}
\end{align}
for some $\mu_j^{(i)}, \mu_{j,ms}^{(i)}\in V_{aux}$. Subtracting the above two equations and restricting $v\in \tilde{V}_0(K_{i,k})$ leads to
\begin{align*}
a(\psi_{j}^{(i)}-\psi_{j,ms}^{(i)},v)=0\quad \forall v\in \tilde{V}_0(K_{i,k}).
\end{align*}
Here, we have $\tilde{V}_0(K_{i,k})=\{v\in H^1_0(K_{i,k})\;|\;\pi(v)=0\}$. Therefore,
for $v\in \tilde{V}_0(K_{i,k})$, we can infer that
\begin{align*}
\|\psi_j^{(i)}-\psi_{j,ms}^{(i)}\|_a^2&
=a(\psi_j^{(i)}-\psi_{j,ms}^{(i)},\psi_j^{(i)}-\psi_{j,ms}^{(i)})\\
&=a(\psi_j^{(i)}-\psi_{j,ms}^{(i)},
\psi_j^{(i)}-\tilde{\phi}_j^{(i)}-\psi_{j,ms}^{(i)}+\tilde{\phi}_j^{(i)})
=a(\psi_j^{(i)}-\psi_{j,ms}^{(i)},\eta-v),
\end{align*}
where $-\psi_{j,ms}^{(i)}+\tilde{\phi}_j^{(i)}\in \tilde{V}_0(K_{i,k})$. Thus, we obtain
\begin{align}
\|\psi_j^{(i)}-\psi_{j,ms}^{(i)}\|_a\leq \|\eta-v\|_a.\label{eq:decay}
\end{align}
To estimate $\|\psi_j^{(i)}-\psi_{j,ms}^{(i)}\|_a$, we need to derive the upper bound for $\|\eta-v\|_a$. We consider the $i$-th coarse
block $K_i$. For this block, we consider two oversampling regions $K_{i,k-1}$ and $K_{i,k}$.
Using these two overampling regions, we define the cutoff function $\chi_i^{k,k-1}$
with the properties in \eqref{cut1}-\eqref{cut2}, where we take $m=k-1$ and $M=k$.
For any coarse block $K_j\subset K_{i,k-1}$,
 we have $\chi_{i}^{k,k-1}\equiv 1$ on $K_j$ by using \eqref{cut1}. Since $\eta\in \tilde{V}$, it holds that
\begin{align*}
s_j(\chi_{i}^{k,k-1}\eta,\phi_n^{(j)})=s_j(\eta, \phi_n^{(j)})=0\quad \forall n =1,2,\cdots,l_j.
\end{align*}
From the above result and the fact that $\chi_i^{k,k-1}\equiv 0$ in $\Omega\backslash K_{i,k}$, we have
\begin{align*}
\mbox{supp}(\pi(\chi_i^{k,k-1}\eta))\subset K_{i,k}\backslash K_{i,k-1}.
\end{align*}

By Lemma~\ref{lemma:vaux}, for the function $\pi(\chi_i^{k,k-1}\eta)$, there is $\mu\in H^1_0(\Omega)$ such that $\mbox{supp}(\mu)\subset K_{i,k}\backslash K_{i,k-1}$ and $\pi(\mu-\chi_i^{k,k-1}\eta)=0$. Moreover, it also follows from Lemma~\ref{lemma:vaux} that
\begin{align}
\|\mu\|_{a(K_{i,k}\backslash K_{i,k-1})}\leq D^{1/2}\|\pi(\chi_i^{k,k-1}\eta)\|_{s(K_{i,k}\backslash K_{i,k-1})}\leq D^{1/2}\|\chi_i^{k,k-1}\eta\|_{s(K_{i,k}\backslash K_{i,k-1})}.\label{eq:muineq}
\end{align}
Hence, taking $v=-\mu+\chi_i^{k,k-1}\eta$ in \eqref{eq:decay}, we can obtain
\begin{align}
\|\psi_j^{(i)}-\psi_{j,ms}^{(i)}\|_a\leq \|\eta-v\|_a\leq \|(1-\chi_i^{k,k-1})\eta\|_a+\|\mu\|_{a(K_{i,k}\backslash K_{i,k-1})}.\label{eq:decay2}
\end{align}
Next, we will estimate the two terms on the right hand side of \eqref{eq:decay2}.

Step 1: We first estimate the first term in \eqref{eq:decay2}. By a direct computation, we have
\begin{align*}
\|(1-\chi_i^{k,k-1})\eta\|_a^2\leq 2\Big(\int_{\Omega\backslash K_{i,k-1}}\kappa(1-\chi_{i}^{k,k-1})^2|\nabla \eta|^2+\int_{\Omega\backslash K_{i,k-1}}\kappa |\nabla \chi_i^{k,k-1}|^2\eta^2\Big).
\end{align*}
Note that, we have $1-\chi_i^{k,k-1}\leq 1$. For the second term on the righ hand side of the above inequality, we will use the fact that $\eta\in \tilde{V}$ and the spectral problem \eqref{eq:spectral}
\begin{align*}
\int_{\Omega\backslash K_{i,k-1}}\kappa |\nabla \chi_i^{k,k-1}|^2\eta^2&\leq C\int_{\Omega\backslash K_{i,k-1}}H^{-2}\kappa |\nabla \eta|^2\leq C \beta_0^{-2} \int_{\Omega\backslash K_{i,k-1}}H^{-2}\kappa |\bm{\beta}|^2 |\nabla \eta|^2=C\beta_0^{-2}\|\eta\|_{s(\Omega\backslash K_{i,k-1})}^2\\
&\leq C\Lambda^{-1}\beta_0^{-2}\|\eta\|_{a(\Omega\backslash K_{i,k-1})}^2.
\end{align*}
Therefore, we can obtain
\begin{align*}
\|(1-\chi_i^{k,k-1})\eta\|_a^2\leq C
(1+\Lambda^{-1})\int_{\Omega\backslash K_{i,k-1}}\kappa |\nabla \eta|^2.
\end{align*}
We will estimate the right hand side in Step 3.

Step 2: In this step we will estimate the second term on the right hand side of \eqref{eq:decay2}. By \eqref{eq:muineq}, the fact that $|\chi_i^{k,k-1}|\leq 1$ and the spectral problem \eqref{eq:spectral}, we have
\begin{align*}
\|\mu\|_{a(K_{i,k}\backslash K_{i,k-1})}^2\leq D \|\chi_i^{k,k-1}\eta\|_{s(K_{i,k}\backslash K_{i,k-1})}^2\leq \frac{D}{\Lambda} \int_{K_{i,k}\backslash K_{i,k-1}}\kappa |\nabla \eta|^2.
\end{align*}

Combining Steps 1 and 2, we obtain
\begin{align}
\|\psi_j^{(i)}-\psi_{j,ms}^{(i)}\|_a^2\leq 2D(1+\frac{1}{\Lambda}) \|\eta\|_{a(\Omega\backslash K_{i,k-1})}^2.\label{eq:step12}
\end{align}

Step 3: Finally, we will estimate the term $\|\eta\|_{a(\Omega\backslash K_{i,k-1})}$. We will first show that the following recursive inequality holds
\begin{align}
\|\eta\|_{a(\Omega\backslash K_{i,k-1})}^2\leq (1+\frac{\Lambda^{1/2}}{2D^{1/2}})^{-1}\|\eta\|_{a(\Omega\backslash K_{i,k-2})}^2,\label{eq:recursive}
\end{align}
where $k-2\geq 0$. Using \eqref{eq:recursive} in \eqref{eq:step12}, we can get
\begin{align}
\|\psi_j^{(i)}-\psi_{j,ms}^{(i)}\|_a^2\leq 2D(1+\frac{1}{\Lambda})(1+\frac{\Lambda^{1/2}}{2D^{1/2}})^{-1}\|\eta\|_{a(\Omega\backslash K_{i,k-2})}^2.\label{eq:step12-recursive}
\end{align}
By using \eqref{eq:recursive} again in \eqref{eq:step12-recursive}, we can obtain
\begin{align*}
\|\psi_j^{(i)}-\psi_{j,ms}^{(i)}\|_a^2&\leq 2D(1+\frac{1}{\Lambda})(1+\frac{\Lambda^{1/2}}{2D^{1/2}})^{1-k}\|\eta\|_{a(\Omega\backslash K_{i})}^2\\
&\leq 2D(1+\frac{1}{\Lambda})(1+\frac{\Lambda^{1/2}}{2D^{1/2}})^{1-k}\|\eta\|_{a}^2.
\end{align*}
By employing the definition of $\eta$, the energy minimizing property of $\psi_j^{(i)}$ and Lemma~\ref{lemma:vaux}, we have
\begin{align*}
\|\eta\|_a=\|\psi_{j}^{(i)}-\tilde{\phi}_j^{(i)}\|_a\leq 2\|\tilde{\phi}_j^{(i)}\|_a\leq 2D^{1/2}\|\phi_j^{(i)}\|_{s(K_i)}.
\end{align*}

Step 4: We will prove the estimate \eqref{eq:recursive}. Let $\xi=1-\chi_i^{k-1,k-2}$. Then we see that $\xi\equiv 1$ in $\Omega\backslash K_{i,k-1}$ and $0\leq \xi\leq 1$ otherwise. Then we have
\begin{align}
\|\eta\|_{a(\Omega\backslash K_{i,k-1})}^2\leq \int_\Omega \kappa \xi^2|\nabla \eta|^2=\int_\Omega \kappa \nabla \eta\cdot \nabla (\xi^2 \eta)-2\int_\Omega \kappa \xi \eta \nabla \xi \nabla \eta.\label{eq:eta}
\end{align}
We estimate the first term in \eqref{eq:eta}. For the function $\pi(\xi^2\eta)$, using Lemma~\ref{lemma:vaux}, there exists $\gamma\in H^1_0(\Omega)$ such that $\pi(\gamma)=\pi(\xi^2\eta)$ and $\mbox{supp}(\gamma)\subset \mbox{supp}(\pi(\xi^2\eta))$. For any coarse elements $K_m\subset \Omega\backslash K_{i,k-1}$, since $\xi\equiv 1$ on $K_m$, we have
\begin{align*}
s_m(\xi^2\eta, \phi_n^{(m)})=0\quad \forall n=1,\ldots, l_m.
\end{align*}
On the other hand, since $\xi\equiv 0$ in $K_{i,k-2}$, there holds
\begin{align*}
s_m(\xi^2\eta, \phi_n^{(m)})=0\quad \forall n=1,\ldots, l_m,\; \forall K_m\subset K_{i,k-2}.
\end{align*}
From the above two conditions, we see that $\mbox{supp}(\pi(\xi^2\eta))\subset K_{i,k-1}\backslash K_{i,k-2}$ and consequently $\mbox{supp}(\gamma)\subset K_{i,k-1} \backslash K_{i,k-2}$. Note that, since $\pi(\gamma)=\pi(\xi^2\eta)$, we have $\xi^2\eta-\gamma \in \tilde{V}$. We also note that $\mbox{supp}(\xi^2\eta-\gamma)\subset \Omega\backslash K_{i,k-2}$. By \eqref{eq:phit}, the functions $\tilde{\phi}_j^{(i)}$ and $\xi^2\eta-\gamma$ have disjoint supports, so $a(\tilde{\phi}_j^{(i)}, \xi^2\eta-\gamma)=0$. Then, by the definition of $\eta$, we have
\begin{align*}
a(\eta, \xi^2\eta-\gamma)=a(\psi_j^{(i)},\xi^2\eta-\gamma).
\end{align*}
By the construction of $\psi_j^{(i)}$, we have $a(\psi_j^{(i)},\xi^2\eta-\gamma)=0$. Then we can estimate the first term in \eqref{eq:eta} by the Cauchy-Schwarz inequality and Lemma~\ref{lemma:vaux}
\begin{align*}
\int_\Omega \kappa \nabla \eta\cdot \nabla(\xi^2\eta)&=\int_\Omega\kappa \nabla \eta\cdot \nabla \gamma\\
&\leq D^{1/2} \|\eta\|_{a(K_{i,k-1}\backslash K_{i,k-2})}\|\pi(\xi^2\eta)\|_{s(K_{i,k-1}\backslash K_{i,k-2})}.
\end{align*}
For all coarse elements $K\subset K_{i,k-1}\backslash K_{i,k-2}$, since $\pi(\eta)=0$, we have
\begin{align*}
\|\pi(\xi^2\eta)\|_{s(K)}^2\leq\|\xi^2\eta\|_{s(K)}^2\leq \frac{1}{\Lambda}\int_K \kappa |\nabla \eta|^2.
\end{align*}
Summing the above over all coarse elements $K\subset K_{i,k-1}\backslash K_{i,k-2}$, we can obtain
\begin{align*}
\|\pi(\xi^2\eta)\|_{s(K_{i,k-1}\backslash K_{i,k-2})}\leq (\frac{1}{\Lambda})^{1/2}  \|\eta\|_{a(K_{i,k-1}\backslash K_{i,k-2})}.
\end{align*}
To estimate the second term in \eqref{eq:eta}, we have from the spectral problem \eqref{eq:spectral}
\begin{align*}
2\int_\Omega \kappa \xi\eta\nabla \xi\cdot\nabla \eta\leq 2 \beta_0^{-1} \| \eta\|_{s(K_{i,k-1}\backslash K_{i,k-2})}\|\eta\|_{a(K_{i,k-1}\backslash K_{i,k-2})}\leq \frac{2}{\beta_0\Lambda^{\frac{1}{2}}}\|\eta\|_{a(K_{i,k-1}\backslash K_{i,k-2})}^2.
\end{align*}
Hence, the preceding arguments yield the upper bound for \eqref{eq:eta}
\begin{align*}
\|\eta\|_{a(\Omega\backslash K_{i,k-1})}^2\leq \frac{2D^{1/2}}{\Lambda^{1/2}}\|\eta\|_{a(K_{i,k-1}\backslash K_{i,k-2})}^2.
\end{align*}
Thus
\begin{align*}
\|\eta\|_{a(\Omega\backslash K_{i,k-2})}^2=\|\eta\|_{a(\Omega\backslash K_{i,k-1})}^2+\|\eta\|_{a(K_{i,k-1}\backslash K_{i,k-2})}^2\geq (1+\frac{\Lambda^{1/2}}{2D^{1/2}})\|\eta\|_{a(\Omega\backslash K_{i,k-1})}^2.
\end{align*}

\end{proof}

Following \cite{ChungEfendievleung18}, we can prove the following lemma. The proof is omitted here for simplicity.
\begin{lemma}\label{lemma:psi}
With the same assumptions as in Lemma~\ref{lemma:decay}, we can obtain
\begin{align*}
\|\sum_{i=1}^N(\psi_j^{(i)}-\psi_{j,ms}^{(i)})\|_a^2&\leq C(k+1)^2\sum_{i=1}^N \|\psi_j^{(i)}-\psi_{j,ms}^{(i)}\|_a^2.
\end{align*}

\end{lemma}




Now we are ready to prove the following theorem, which gives an estimate of the error between the weak solution $u$ and the multiscale solution $u_{ms}$.
\begin{theorem}

Let $u$ be the solution of \eqref{eq:weak} and $u_{ms}$ be the solution of \eqref{eq:multiscale-solution}. Then, we have
\begin{align*}
\|u-u_{ms}\|_a&\leq C\Big( (1+H\kappa_0^{-1}\Lambda^{-\frac{1}{2}})H\Lambda^{-\frac{1}{2}}\Big(
\|\kappa^{-\frac{1}{2}} |\bm{\beta}|^{-1} f\|_{L^2(\Omega)}\\
&\;+\|\kappa^{-\frac{1}{2}}\nabla u\|_{L^2(\Omega)}\Big)+(1+H\kappa_0^{-1}\Lambda^{-\frac{1}{2}} )(1+k)E^{\frac{1}{2}}\|\tilde{u}\|_s\Big).
\end{align*}
where $\tilde{u}$ is the projection obtained from \eqref{eq:projection}. Moreover, if the number of oversampling layers $k=\mathcal{O}(log(\frac{\beta_1\kappa_1^{\frac{1}{2}}}{H\kappa_0^{\frac{1}{2}}}))$ and $H\kappa_0^{-1}\Lambda^{-\frac{1}{2}} \leq C $,
we have
%
\begin{align*}
\|u-u_{ms}\|_a\leq CH \Lambda^{-\frac{1}{2}}\Big(
\|\kappa^{-\frac{1}{2}} |\bm{\beta}|^{-1} f\|_{L^2(\Omega)}+\|\kappa^{-\frac{1}{2}}\nabla u\|_{L^2(\Omega)}\Big).
\end{align*}

\end{theorem}

\begin{proof}

Since $\bm{\beta}$ is divergence free, it holds that
\begin{equation*}
\|u-u_{ms}\|_a^2 = a(u-u_{ms},u-u_{ms})
= a(u-u_{ms},u-u_{ms}) + \int_{\Omega} \bm{\beta}\cdot \nabla(u-u_{ms}) \, (u-u_{ms}).
\end{equation*}
For any $v\in V_{H}$, we can infer from Galerkin orthogonality that
\begin{equation}
\begin{split}
\|u-u_{ms}\|_a^2& = a(u-u_{ms},u-v) + \int_{\Omega} \bm{\beta}\cdot \nabla(u-u_{ms}) \, (u-v)\\
&\leq \|u-u_{ms}\|_a\|u-v\|_a+H\kappa_0^{-1}\|u-u_{ms}\|_a\|u-v\|_s.
\end{split}
\label{eq:uH}
\end{equation}
We assume that $\tilde{u}$ obtained from \eqref{eq:projection} can be written as $\tilde{u}=\sum_{i=1}^N\sum_{j=1}^{l_i}c_j^{(i)}\psi_j^{(i)}$. Then we define a function $v=\sum_{i=1}^N\sum_{j=1}^{l_i}c_j^{(i)}\psi_{j,ms}^{(i)}\in V_{ms}$.
We can infer from the spectral problem \eqref{eq:spectral} that
\begin{align*}
\|u-v\|_s\leq \|u-\tilde{u}\|_s+\|\tilde{u}-v\|_s\leq \Lambda^{-1/2} \Big(\|u-\tilde{u}\|_a+\|\tilde{u}-v\|_a\Big),
\end{align*}
which coupling with \eqref{eq:uH} yields
\begin{align*}
\|u-u_{ms}\|_a&\leq \|u-v\|_a+H\kappa_0^{-1}\Lambda^{-\frac{1}{2}} \Big(\|u-\tilde{u}\|_a+\|\tilde{u}-v\|_a\Big)\\
&\leq \|u-\tilde{u}\|_a+\|\tilde{u}-v\|_a+H\kappa_0^{-1}\Lambda^{-\frac{1}{2}} \Big(\|u-\tilde{u}\|_a+\|\tilde{u}-v\|_a\Big).
\end{align*}
Then it follows from Lemma~\ref{lemma:projection} that
\begin{align*}
\|u-u_{ms}\|_a&\leq (1+H\kappa_0^{-1}\Lambda^{-\frac{1}{2}} )\|u-\tilde{u}\|_a+(1+H\kappa_0^{-1}\Lambda^{-\frac{1}{2}} )\|\tilde{u}-v\|_a\\
&\leq \Big((1+H\kappa_0^{-1}\Lambda^{-\frac{1}{2}})H\Lambda^{-\frac{1}{2}}(
\|\kappa^{-\frac{1}{2}} |\bm{\beta}|^{-1} f\|_{L^2(\Omega)}+\|\kappa^{-\frac{1}{2}}\nabla u\|_{L^2(\Omega)})\\
&\;+(1+H\kappa_0^{-1}\Lambda^{-\frac{1}{2}} )\|\tilde{u}-v\|_a\Big).
\end{align*}
We can infer from Lemmas~\ref{lemma:decay} and \ref{lemma:psi} that
\begin{align*}
\|\tilde{u}-v\|_a^2&=\|\sum_{i=1}^N\sum_{j=1}^{l_i}c_j^{(i)}(\psi_j^{(i)}-\psi_{j,ms}^{(i)})\|_a^2\leq C (1+k)^2\sum_{i=1}^N\|\sum_{j=1}^{l_i}c_j^{(i)}(\psi_j^{(i)}-\psi_{j,ms}^{(i)})\|_a^2\\
&\leq C (1+k)^2E\sum_{i=1}^N\|\sum_{j=1}^{l_i}c_j^{(i)}\phi_j^{(i)}\|_s^2\\
&\leq C (1+k)^2E\|\tilde{u}\|_s^2.
\end{align*}
Therefore
\begin{align*}
\|u-u_{ms}\|_a&\leq C\Big( (1+H\kappa_0^{-1}\Lambda^{-\frac{1}{2}})H\Lambda^{-\frac{1}{2}}(
\|\kappa^{-\frac{1}{2}} |\bm{\beta}|^{-1} f\|_{L^2(\Omega)}+\|\kappa^{-\frac{1}{2}}\nabla u\|_{L^2(\Omega)})\\
&\;+(1+H\kappa_0^{-1}\Lambda^{-\frac{1}{2}} )(1+k)E^{\frac{1}{2}}\|\tilde{u}\|_s\Big).
\end{align*}
%
Now we show the error estimate for $\|\tilde{u}\|_s$. By the definition of $s$-norm, we have
\begin{align}
\|\tilde{u}\|_s^2= H^{-2}\int_{\Omega}\kappa |\bm{\beta}|^2|\tilde{u}|^2\;dx\leq H^{-2}\kappa_1\beta_1^2\|\tilde{u}\|_{L^2(\Omega)}^2\leq H^{-2}\kappa_1\kappa_0^{-1}\beta_1^2\|\tilde{u}\|_a^2\label{eq:tildeus}.
\end{align}
Recall that $\nabla\cdot \bm{\beta}=0$, thereby integration by parts yields $(\bm{\beta}\cdot\nabla \tilde{u}, \tilde{u})=0$. Hence, we have
\begin{align*}
\|\tilde{u}\|_a^2=(\kappa\nabla\tilde{u},\nabla \tilde{u})=(f,\tilde{u})-(\bm{\beta}\cdot\nabla \tilde{u}, \tilde{u})=(f,\tilde{u})\leq H\|\kappa^{-\frac{1}{2}}|\bm{\beta}|^{-1}f\|_{L^2(\Omega)}\|\tilde{u}\|_s,
\end{align*}
which together with \eqref{eq:tildeus} yields
\begin{align*}
\|\tilde{u}\|_s\leq H^{-2}\kappa_1\kappa_0^{-1}\beta_1^2 H \|\kappa^{-\frac{1}{2}} |\bm{\beta}|^{-1} f\|_{L^2(\Omega)}\leq H^{-1}\kappa_1\kappa_0^{-1} \beta_1^2 \|\kappa^{-\frac{1}{2}} |\bm{\beta}|^{-1} f\|_{L^2(\Omega)}.
\end{align*}
In order to deliver convergent solution, we require that $H\kappa_0^{-1}\Lambda^{-\frac{1}{2}} \leq C $ and  $H^{-1}\beta_1^2(1+k)E^{\frac{1}{2}}\kappa_1\kappa_0^{-1}$ is bounded, i.e.,
\begin{align*}
H^{-1}\beta_1^2(1+k)E^{\frac{1}{2}}\kappa_1\kappa_0^{-1}=\mathcal{O}(1).
\end{align*}
Thus, we need to take $k=\mathcal{O}(log(\frac{\beta_1\kappa_1^{\frac{1}{2}}}{H\kappa_0^{\frac{1}{2}}}))$. Therefore, the proof is completed.

\end{proof}

%
%
%
%
%
%
%
%
%

\section{Numerical experiments}\label{sec:numerical}
In this section we present several numerical experiments to test the performances of our method. Specially, we will study the influences of the number of oversampling layers and the number of basis functions on the error of the multiscale solution. Then the convergence behavior with respect to coarse meshsize $H$ will also be investigated.
In the following examples, we use $N_{ov}$ to denote the number of oversampling layers and $N_b$ to denote the number of basis functions chosen in the spectral problem. In addition, the fine meshsize is defined to be $\frac{\sqrt{2}}{400}$. In the simulations given below, we take $\Omega=(0,1)^2$. In addition, we define the following errors for later use.
\begin{align*}
e_{L^2}=\frac{\|u_h-u_{ms}\|_{L^2(\Omega)}}{\|u_h\|_{L^2(\Omega)}},\quad e_{H^1}=\frac{\|\kappa^{1/2}\nabla(u_h-u_{ms})\|_{L^2(\Omega)}}{\|\kappa^{1/2}\nabla u_h\|_{L^2(\Omega)}},
\end{align*}
where $u_h$ is the fine scale solution obtained from standard conforming finite element method.

\subsection{Example 1}\label{ex1}

In our first example we consider $\kappa=1/200$, the velocity field is given by
\begin{align*}
\bm{\beta}=(\cos(18\pi y)\sin(18\pi x), -\cos(18\pi x)\sin(18\pi y))^T
\end{align*}
and the source term $f$ is given by $f=1$. The fine scale solution and downscale solution with $H=1/20, N_{ov}=3, N_b=5$ are depicted in Figure~\ref{ex1:solution}. We study the effects of the number of oversampling layers and the number of basis functions, and the results are plotted in Figure~\ref{fig:ex-errors}. It can be observed that the accuracy will get better as the number of oversampling layers and the number of basis functions increase. Further, when enough number of oversampling layers and basis functions are given, the error tends to be a constant. Then we show the convergence behavior with respect to the coarse meshsize and the results are reported in Table~\ref{table:con}. We can see that the sequence of solutions converge as the coarse meshsize converges.

%
%
%
%
%
%
%
%
%
%
%
%
%
%
%
%
%
%
%
%


%
%
%
%
%
%
%
%
%

\begin{figure}[t]
\centering
\includegraphics[width=0.32\textwidth]{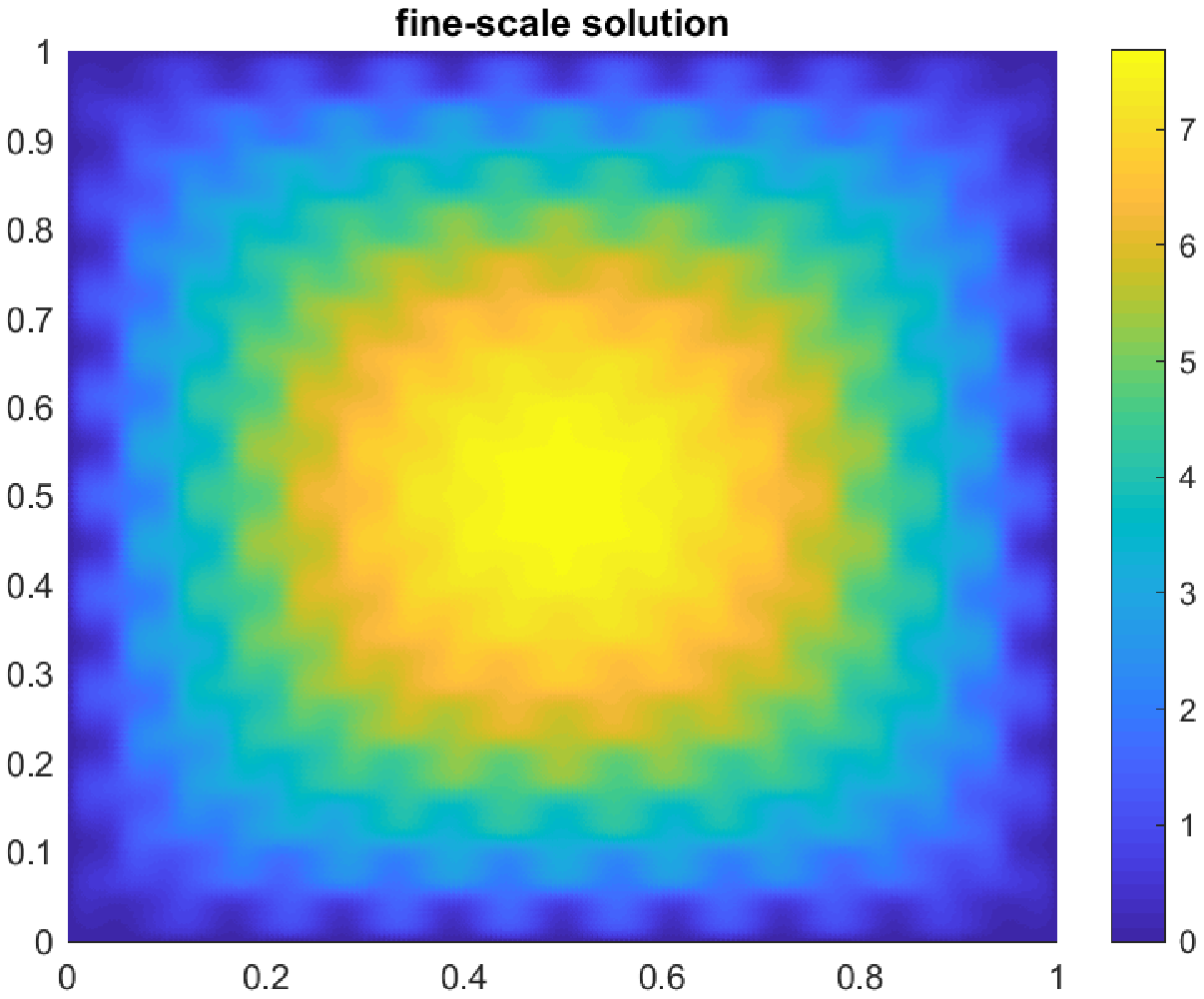}
\includegraphics[width=0.32\textwidth]{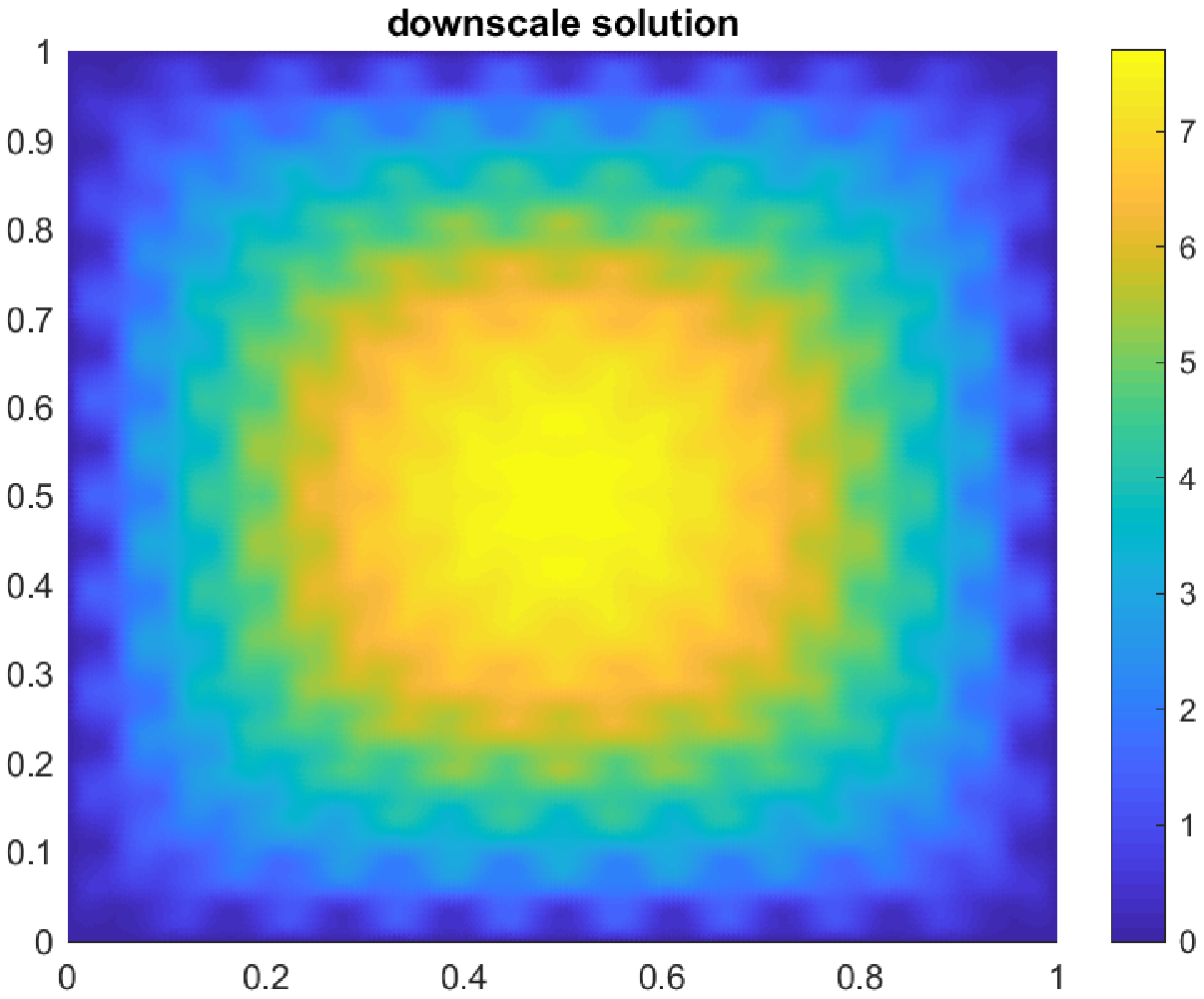}
\caption{Example~\ref{ex1}. Profile of fine scale solution (left) and downscale solution with $H=1/20, N_{ov}=3, N_b=5$ (right).}
\label{ex1:solution}
\end{figure}

\begin{figure}[t]
\centering
\includegraphics[width=0.32\textwidth]{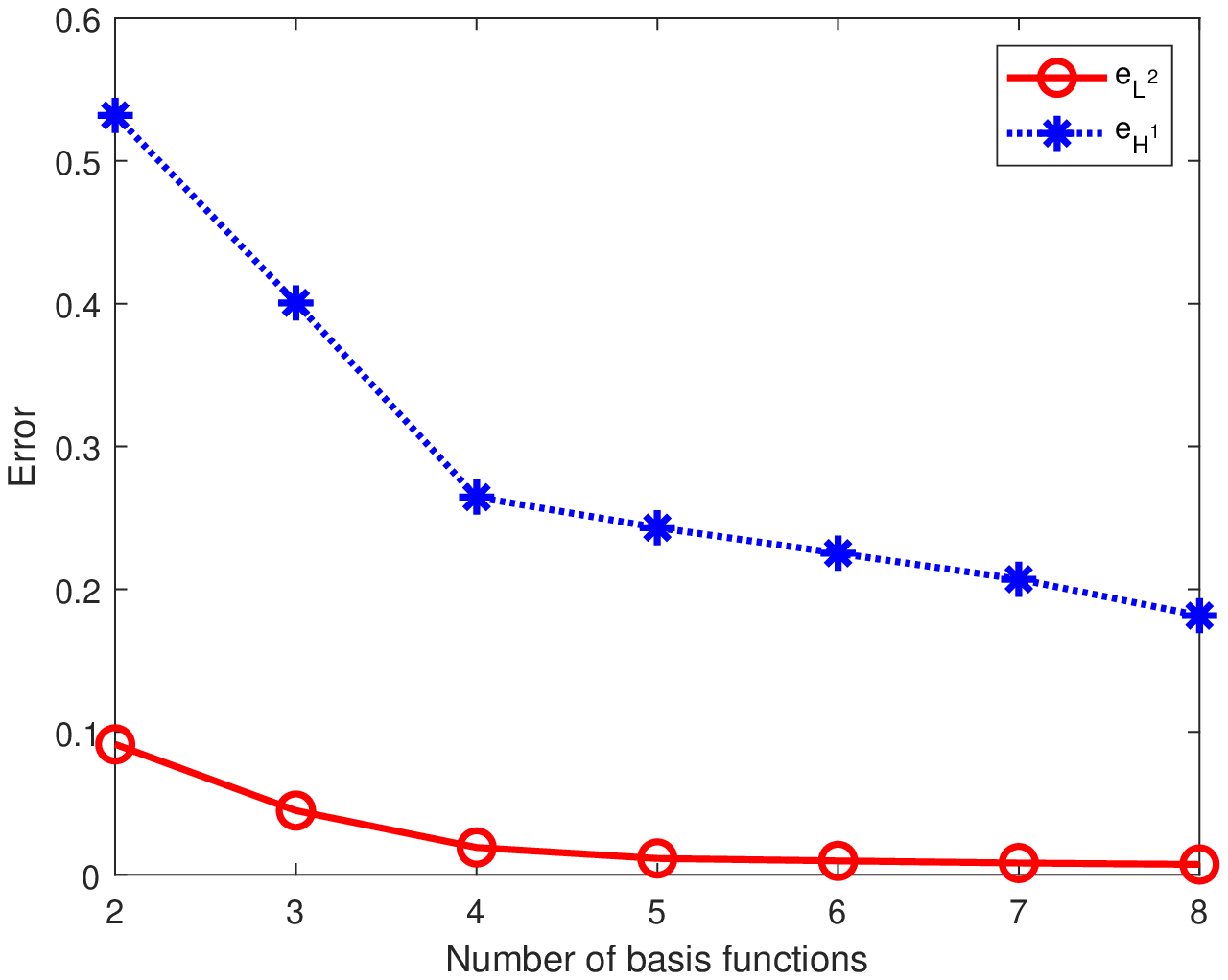}
\includegraphics[width=0.32\textwidth]{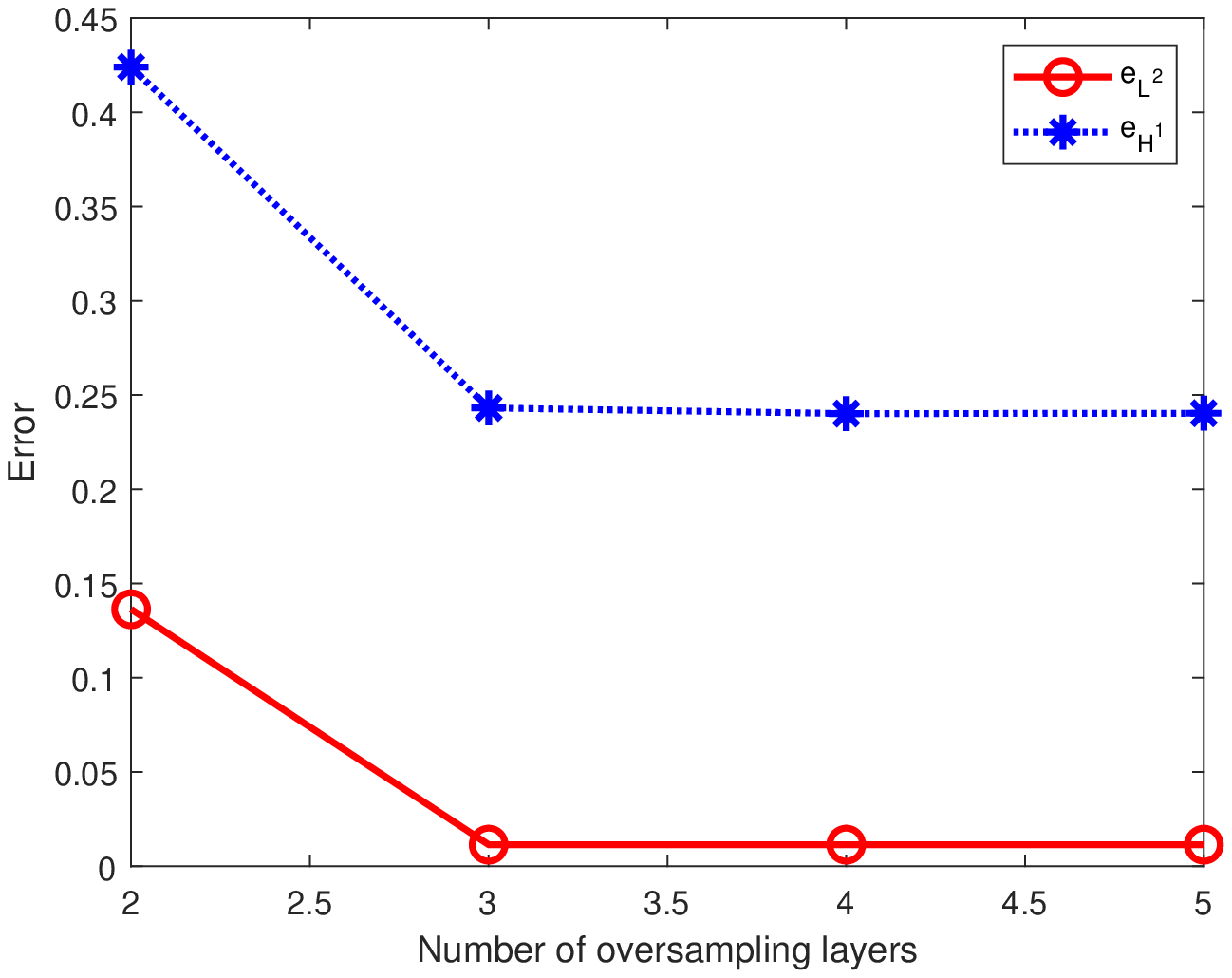}
\caption{Example~\ref{ex1}. Errors with different numbers of basis functions, $H=1/10$, $N_{ov}=2$ (left) and errors with different number of oversampling layers, $H=1/20$, $N_b=5$ (right).}
\label{fig:ex-errors}
\end{figure}

\begin{table}[t]
\centering
\begin{tabular}{lllll}
\hline
$N_b$& H & $N_{ov}$ & $e_{L^2}$ & $e_{H^1}$\\
\hline
5 & 1/10 & 2  &0.1055 & 0.6244 \\
5 & 1/20 & 3  &0.0114  & 0.2431\\
5 & 1/40 & 4  &0.0022  & 0.0846 \\
\hline
\end{tabular}
\caption{Convergence behavior for Example~\ref{ex1}.}
\label{table:con}
\end{table}

\subsection{Example 2}\label{ex2}

In this example we take $\bm{\beta}=\mbox{curl}(b),b=-\sin(80 \pi x)\cos(40\pi y)-250 x+150 y$ and $\kappa=1$. The source term $f$ is defined by
%
%
%
%
%
%
%
%
%
%
%
%
%
%
%
%
%
%
%
%
%
%
%
\begin{align}
f=
\begin{cases}
1\quad 0\leq x,y\leq 0.1,\\
0\quad \mbox{otherwise}.
\end{cases}
\label{eq:f}
\end{align}

Figure~\ref{fig:ex2-sol} shows the fine scale solution and downscale solution with $H=1/20, N_{ov}=3, N_b=5$. The effects of the number of oversampling layers and the number of basis functions on the relative $L^2$ error and relative $H^1$ error are displayed in Figure~\ref{fig:ex2-error}. Similarly, we can observe that increasing the number of oversampling layers and the number of basis functions will decrease the error, also when $N_b$ and $N_{ov}$ are big enough, the error will not decrease anymore. Finally, we show the convergence behavior in Table~\ref{table:ex3-con} and we can observe similar performances to Example~\ref{ex1}.


\begin{figure}[t]
\centering
\includegraphics[width=0.32\textwidth]{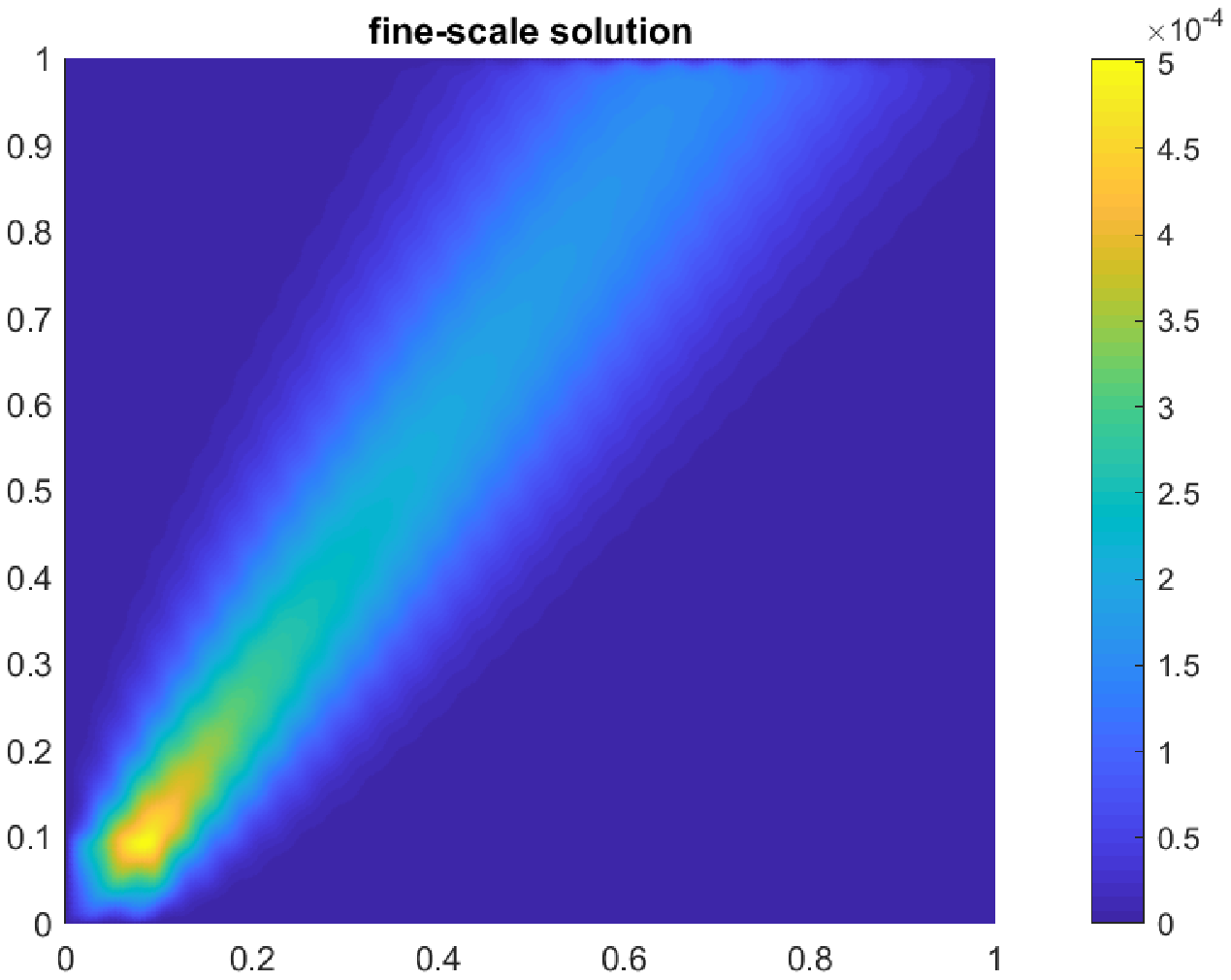}
\includegraphics[width=0.32\textwidth]{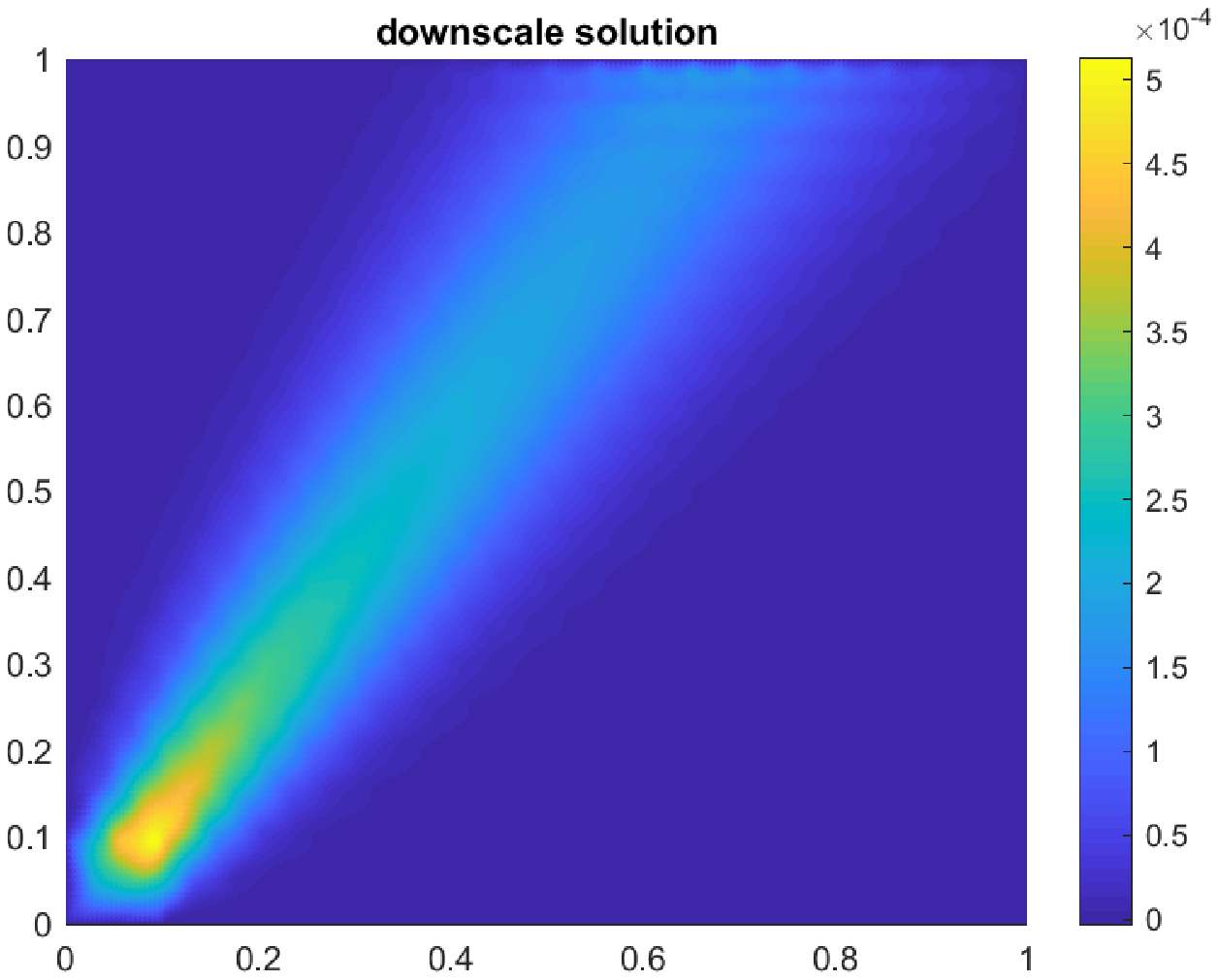}
\caption{Example~\ref{ex2}. Profile of fine scale solution (left) and downscale solution with $H=1/20, N_{ov}=3, N_b=5$ (right).}
\label{fig:ex2-sol}
\end{figure}

\begin{figure}[t]
\centering
\includegraphics[width=0.32\textwidth]{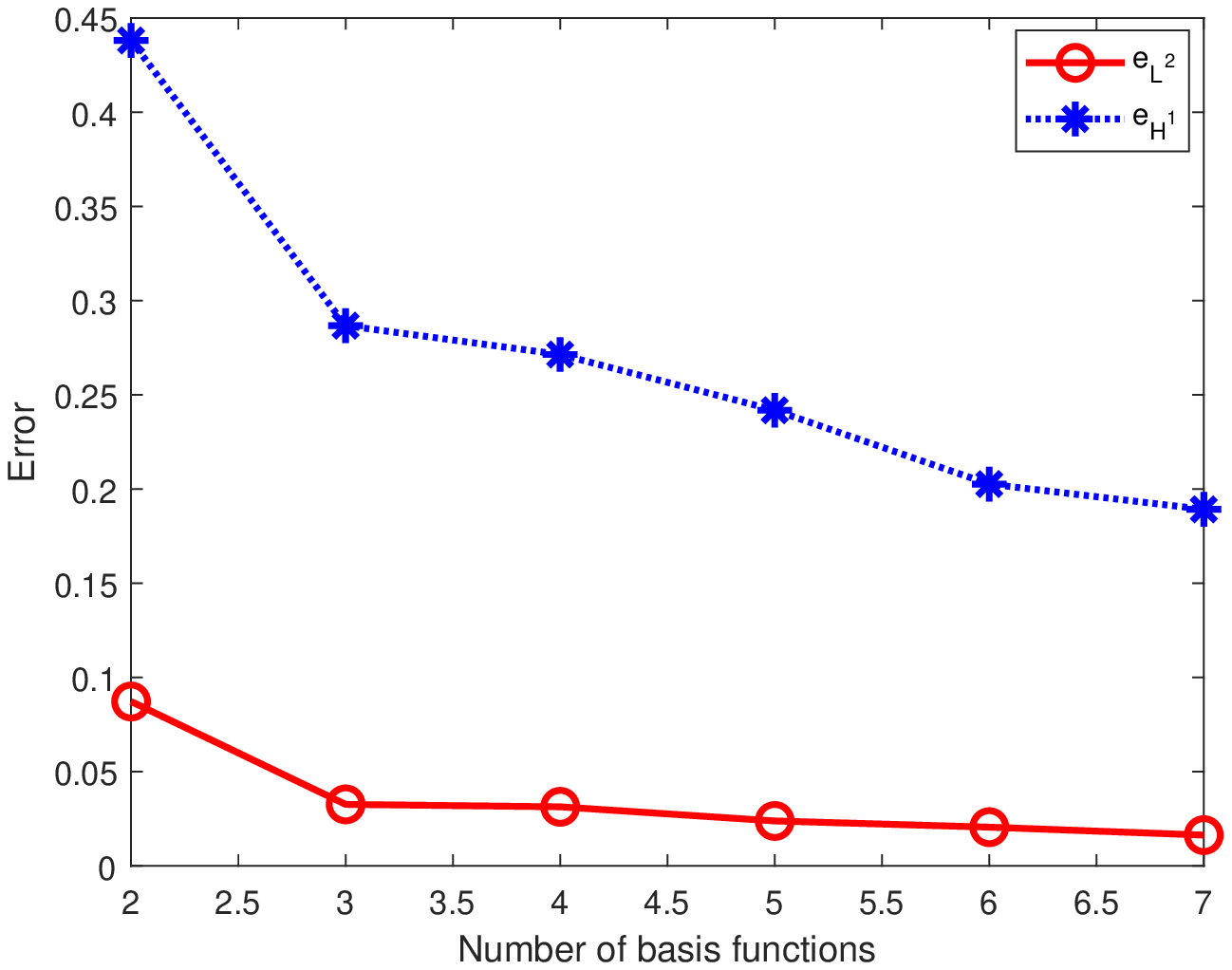}
\includegraphics[width=0.32\textwidth]{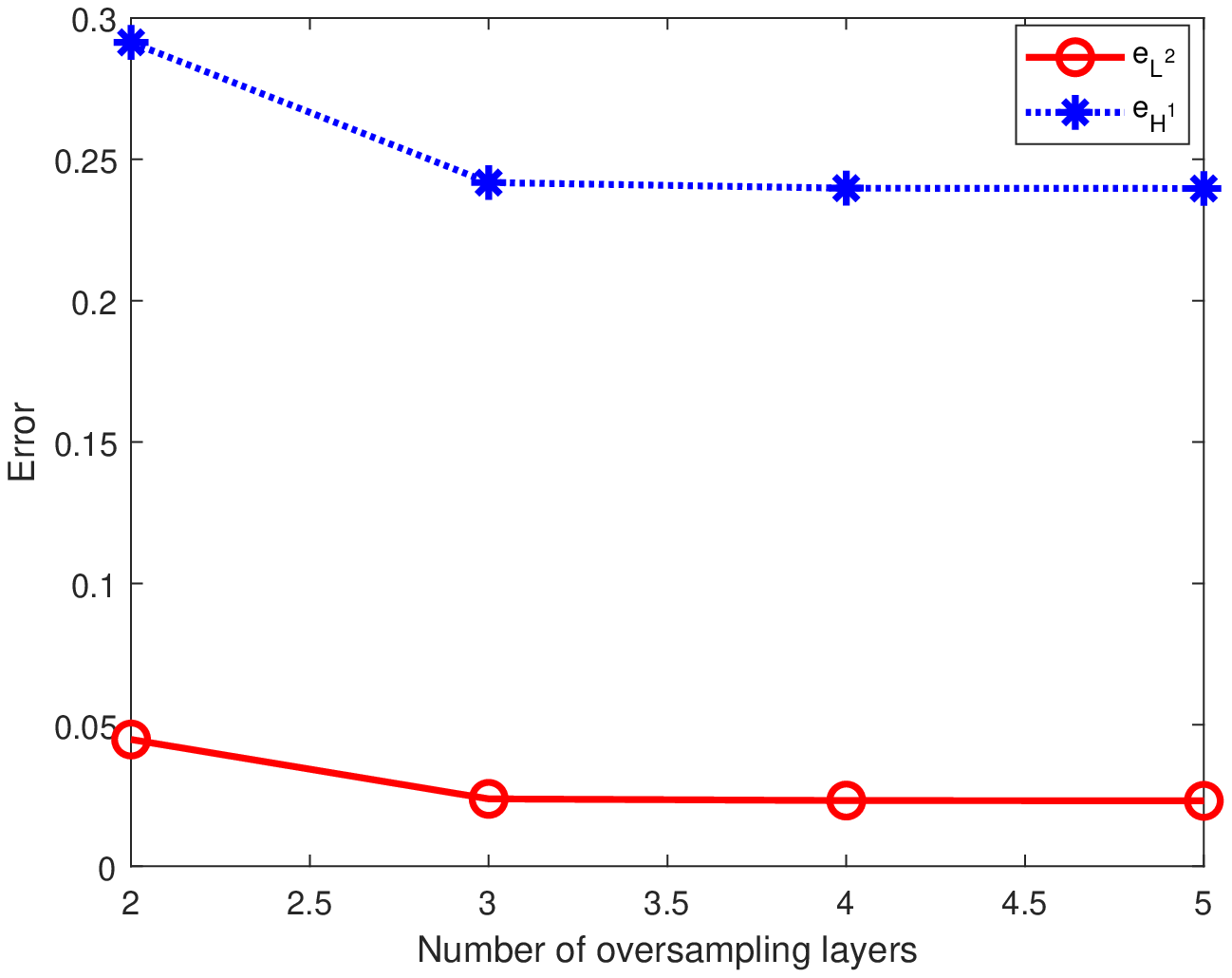}
\caption{Example~\ref{ex2}. Errors with different numbers of basis functions, $H=1/20$, $N_{ov}=3$ (left) and errors with different numbers of oversampling layers, $H=1/20$, $N_b=5$ (right).}
\label{fig:ex2-error}
\end{figure}

\begin{table}[t]
\centering
\begin{tabular}{lllll}
\hline
$N_b$& H & $N_{ov}$ & $e_{L^2}$ & $e_{H^1}$\\
\hline
5 & 1/10 & 2  &0.0928  &  0.4183\\
5 & 1/20 & 3  &0.0238  &  0.2418\\
5 & 1/40 & 4  &0.0030  &  0.0725\\
\hline
\end{tabular}
\caption{Convergence behavior for Example~\ref{ex2}.}
\label{table:ex3-con}
\end{table}

\subsection{Example 3}\label{ex:SPE}
In this example we consider a more challenging case, where the velocity field $\bm{\beta}$ is determined by a Darcy flow in a high contrast medium.
In particular, the velocity field $\bm{\beta}$ and pressure field $p$ is determined by the following system:
\begin{align}
\begin{alignedat}{2}\label{eq:darcy}
\bm{\beta}&=-K \nabla p&& \quad \mbox{in}\;\Omega, \\
\nabla\cdot \bm{\beta} &=q&&\quad \mbox{in}\;\Omega,\\
p& =0&& \quad \mbox{on}\;\partial \Omega,
\end{alignedat}
\end{align}
where we set
\begin{align*}
q=
\begin{cases}
1\qquad 0\leq x,y\leq 0.1,\\
-1\quad   0.9 \leq x,y\leq 1,\\
0\qquad \mbox{otherwise}
\end{cases}
\end{align*}
and $f$ is defined to be the same as in \eqref{eq:f}. Here the heterogeneous field $K$ is defined by a SPE benchmark case \cite{Christie01} as is shown in Figure~\ref{ex:SPE}. We solve \eqref{eq:darcy} by using a lowest order Raviart-Thomas mixed finite element method and the profile for $\bm{\beta}$ is depicted in Figure~\ref{ex:SPE}.
In this case we set $\kappa=0.02$. The convergence behavior is shown in Table~\ref{Table:ex3-con}, and we can see that the sequences of solution converge as the meshsize gets smaller.

\begin{figure}[t]
\centering
\includegraphics[width=0.35\textwidth]{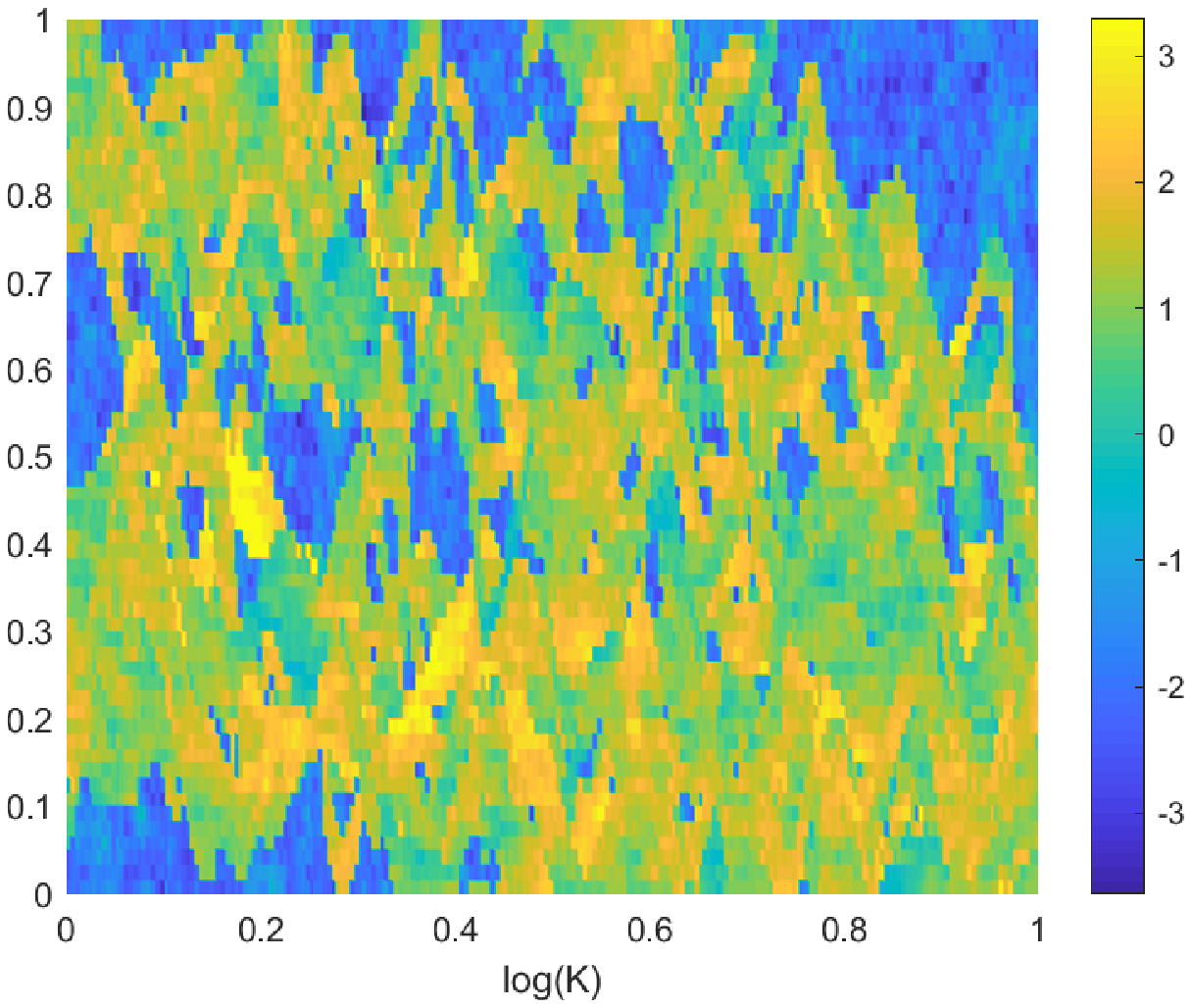}
\includegraphics[width=0.35\textwidth]{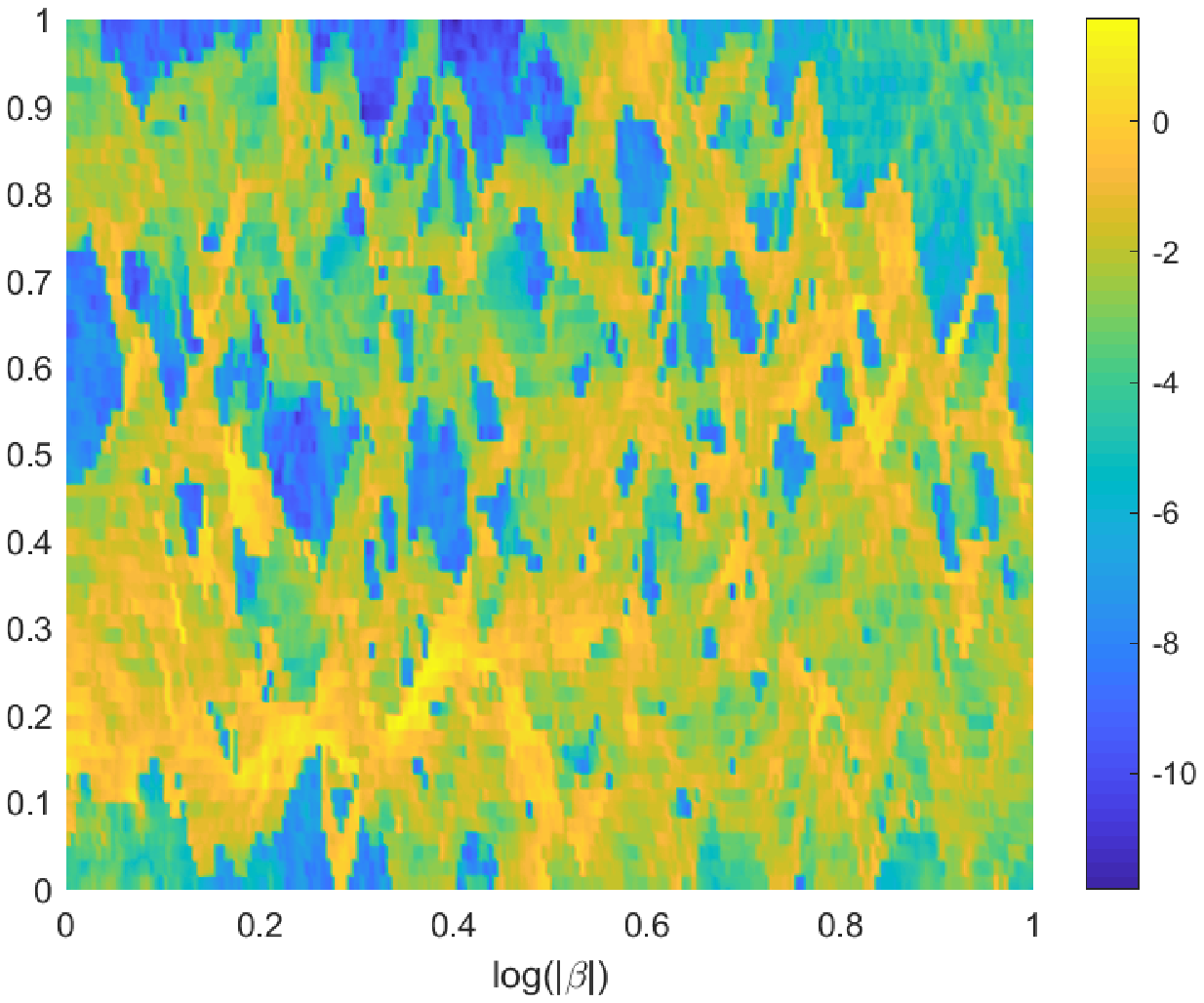}
\caption{Example~\ref{ex:SPE}. Profile of $K$ (left) and $|\bm{\beta}|$ (right) in log scale.}
\end{figure}

\begin{figure}[t]
\centering
\includegraphics[width=0.35\textwidth]{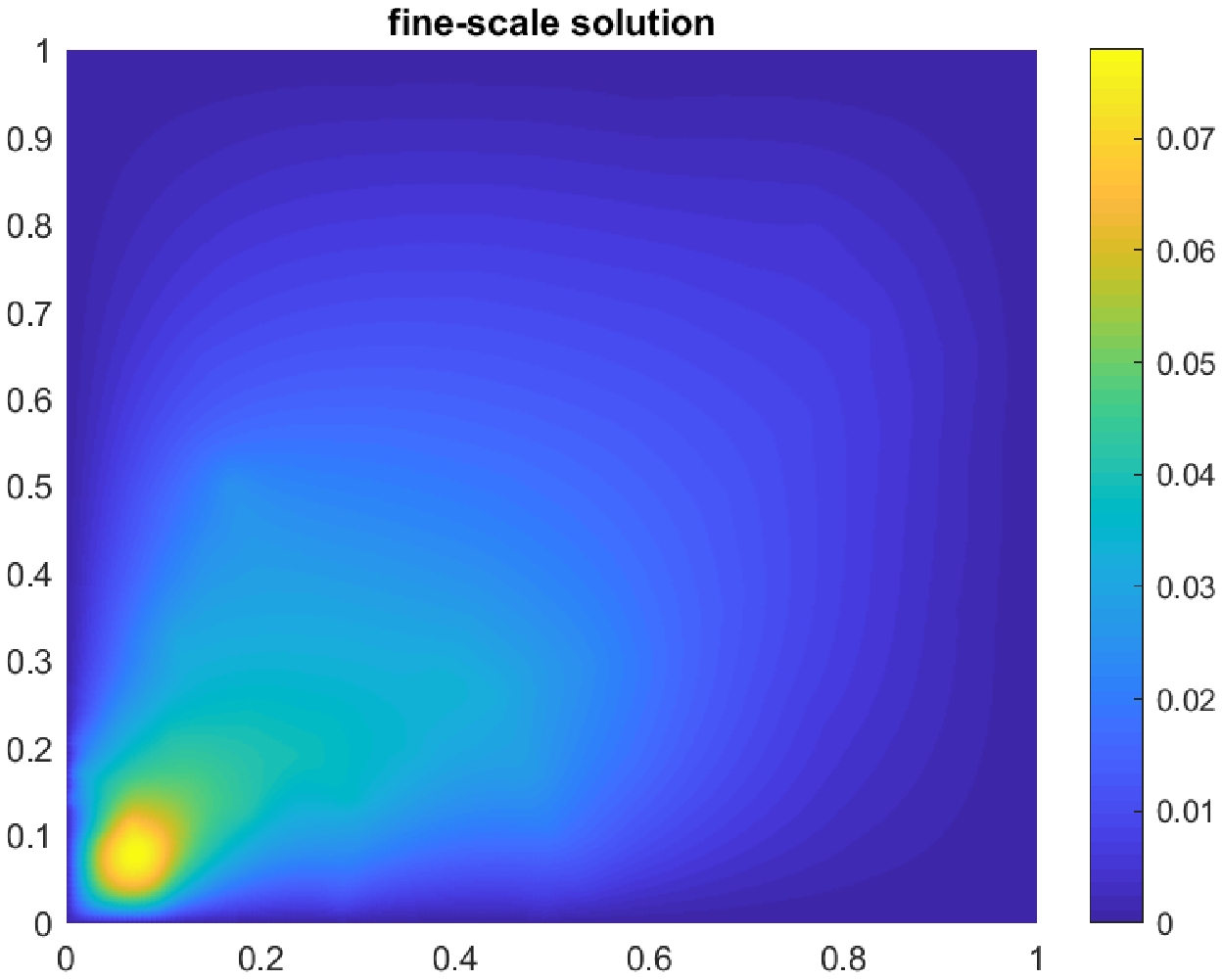}
\includegraphics[width=0.35\textwidth]{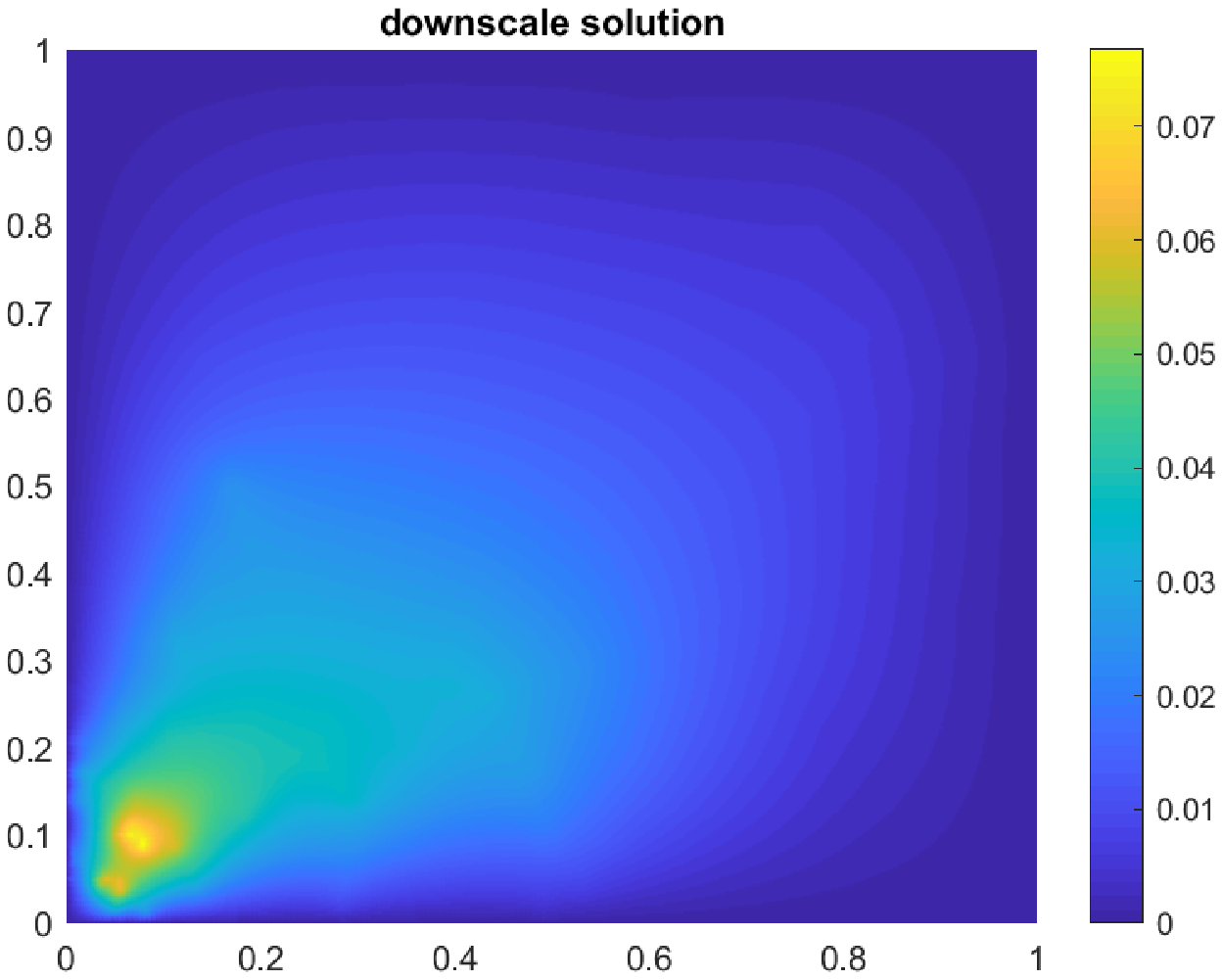}
\caption{Example~\ref{ex:SPE}. Profile of fine scale solution  (left) and downscale solution with $H=1/20, N_{ov}=5, N_{b}=5$ (right).}
\label{Table:ex3-con}
\end{figure}


\begin{table}[t]
\centering
\begin{tabular}{lllll}
\hline
$N_b$& H & $N_{ov}$ & $e_{L^2}$ & $e_{H^1}$\\
\hline
5 & 1/10 & 3  &0.1809  &  0.5990\\
5 & 1/20 & 4  &0.0679  &  0.3393\\
5 & 1/40 & 5  &0.0237  &  0.1503\\
\hline
\end{tabular}
\caption{Convergence behavior for Example~\ref{ex:SPE}.}
\label{SPE}
\end{table}

\section{Conclusion}\label{sec:conclusion}

In this paper we have developed constraint energy minimizing generalized multiscale finite element method for convection diffusion equation. The decay property of the multiscale basis functions is proved. In line of this, we prove the convergence of the multiscale solution. Our theories indicate that if the overampling layer is taken properly, then the resulting multiscale basis functions have a decay property. Several numerical experiments are presented to verify the performances of our method. In the future we aim to develop a novel method in the framework of CEM-GMsFEM to solve convection dominated diffusion problem that exhibits interior or boundary layers.

\section*{Acknowledgment}

The research of Eric Chung is partially supported by the Hong Kong RGC General Research Fund (Project numbers 14304719 and 14302620) and CUHK Faculty of Science Direct Grant 2020-21.

\bibliographystyle{plain}
\bibliography{reference}

\end{document}